\documentclass{article}
\usepackage[english]{babel}

\usepackage[letterpaper,top=2cm,bottom=2cm,left=3cm,right=3cm,marginparwidth=1.75cm]{geometry}

\usepackage[utf8]{inputenc} % allow utf-8 input
\usepackage[T1]{fontenc}    % use 8-bit T1 fonts
\usepackage{url}            % simple URL typesetting
\usepackage{booktabs}       % professional-quality tables
\usepackage{amsfonts}       % blackboard math symbols
\usepackage{nicefrac}       % compact symbols for 1/2, etc.
\usepackage{microtype}      % microtypography
\usepackage{xcolor}         % colors

\usepackage{amsmath}
\usepackage{amsthm}
\usepackage{graphicx}
\usepackage[colorlinks=true, allcolors=blue]{hyperref}
\usepackage{bm}

\newtheorem{theorem}{Theorem}[section]
\newtheorem{corollary}[theorem]{Corollary}
\newtheorem{proposition}[theorem]{Proposition}
\newtheorem{lemma}[theorem]{Lemma}
\newtheorem{remark}[theorem]{Remark}
\newtheorem{definition}[theorem]{Definition}
\newtheorem{assumption}[theorem]{Assumption}
\usepackage{algorithm}
\usepackage{algpseudocode}
\usepackage{comment}
\allowdisplaybreaks
\usepackage{natbib}

\title{Consensus Planning  with Primal, Dual, and Proximal Agents}
\author{Alvaro Maggiar\\
 Amazon Supply Chain Optimization Technologies \\
 \texttt{maggiara@amazon.com}
 \and
  Lee Dicker\\
  Amazon Supply Chain Optimization Technologies \\
  \texttt{leehd@amazon.com}
  \and
  Michael W. Mahoney\\
  Amazon Supply Chain Optimization Technologies \\
  \texttt{zmahmich@amazon.com}
 } 
\date{October 21, 2025}

\begin{document}

\maketitle

\begin{abstract}
\noindent 
Consensus planning is a method for coordinating decision making across complex systems and organizations, including complex supply chain optimization pipelines.
It arises when large interdependent distributed agents (systems) share common resources and must act in order to achieve a joint goal.  
In prior consensus planning work, all agents have been assumed to have the same interaction pattern (e.g., all dual agents or all primal agents or all proximal agents), most commonly using the Alternating Direction Method of Multipliers (ADMM) as proximal agents.
However, this is often not a valid assumption in practice, where agents consist of large complex systems, and where we might not have the luxury of modifying these large complex systems at will.  
In this paper, we introduce a consensus algorithm that overcomes this hurdle by allowing for the coordination of agents with different types of interfaces (named primal, dual, and proximal). 
Our consensus planning algorithm allows for any mix of agents by combining ADMM-like updates for the proximal agents, dual ascent updates for the dual agents, and linearized ADMM updates for the primal agents.
We prove convergence results for the algorithm, namely a sublinear $O(1/k)$ convergence rate under mild assumptions, and two-step linear convergence under stronger assumptions. 
We also discuss enhancements to the basic method and provide illustrative empirical results.
\end{abstract}

\section{Introduction}

\paragraph{Consensus Planning.}
Consensus planning refers to a  coordination mechanism to align different distributed agents (or systems) who share common resources and who must act in order to achieve a joint goal. 
For example, in the context of a large retailer, such a distribution of  tasks is necessary, given the scale of the supply chain problem, with separate systems focusing on (say) buying, removals, placement, capacity control, fulfillment, or transfers, to name just a few.  
These individual components act semi-independently, but some form of communication is required, given their interdependence, to ensure that they work in unison. 
This alignment is often times built through ad-hoc solutions, or proxies of connected systems.  To overcome these limitations, distributed, market-based approaches,  such as a Consensus Planning Protocol (CPP) \citep{cpp,cpptalk}, can be developed within organizations. A CPP starts with the agents, and it introduces a coordinator.
Information is exchanged between the agents through the coordinator in a structured manner; and the coordinator reconciles the various favored plans of the individual agents in a principled way to eventually reach a consensus. 
Associated with such mechanisms are agent-level plans, costs that characterize their favored plans, and costs of deviating from them.

\paragraph{Problem.} 
Consensus planning is often viewed as a variant of distributed optimization, where there are relatively few optimization agents and querying information from each agent is comparatively costly.  
A CPP relies on the iterative exchange of information between agents and a coordinator.  
The mechanism through which the coordinator updates and aligns the plans and prices hinges upon the nature of the agents.  
A theoretically popular approach centers around the use of \emph{proximal} interfaces, which are motivated by the theoretically attractive properties of the Alternative Direction Method of Multipliers (ADMM) \citep{boyd2011distributed}.  
In this proximal framework, the agents and the coordinator exchange two types of information:
1) a favored plan (primal variable); and 
2) a price (or cost of deviation) (dual variable).  
These two pieces of information are iteratively exchanged and updated by the agents and coordinator until convergence.  
Implicit in this approach of using ADMM as a coordination mechanism is the assumption that all agents are amenable to interacting with the coordinator via a proximal~interface. 

In many practical industrial settings, however, this assumption is not satisfied. 
The reason is that large complex systems already in place cannot be readily modified to fit this framework without incurring large engineering costs and development time. 
One example of such a hurdle would be in the presence of systems designed as large linear or mixed integer programs that would not allow for the addition of a proximal (nonlinear) term in their objective function.  
Another example is when a system that knows its utility function and can output the value and the gradient of that function (given a proposed plan) interacts with an agent that takes prices as input and outputs a plan but can not output the value of its utility function.

More generally, we can categorize agents into three broad classes, depending on the interface that they support: primal, dual, and proximal.  
Primal agents can be probed given a tentative plan, and they return a corresponding cost of deviation; 
dual agents can be called with a price, and they return a corresponding favored plan; and 
proximal agent can consume both a tentative plan and tentative cost, and they return updated values thereof. 
Most CPP algorithms are built with one type of agent in mind, requiring that all agents belong to the same class. 
The practical difficulty that we address theoretically in this paper is to devise methodologies that allow for a ``mix-and-match'' of agents, and are compatible with all types of agent interfaces, so as not to require the preliminary burden of modifying them to fit the same~type.
 
\paragraph{Main contributions.}
The purpose of this paper is to propose a generic CPP algorithm that is compatible with all three type of agents, without needing that any agent be modified. 
Our main algorithm combines ADMM-like updates for the proximal agents, dual ascent updates for the dual agents, and linearized ADMM updates for the primal agents. 
We prove several convergence results, including sublinear $O\left(\frac{1}{k}\right)$ convergence rate under mild conditions, and two-step linear convergence under stronger conditions; and we provide numerical examples that illustrate the behavior of the algorithm under different combinations of agent types, as well as the impact of acceleration. 
\section{Background}\label{sec:background}

\subsection{Consensus Problem}

The consensus problem is one that involves a set of agents $\mathcal{M}$ that each have their own cost function $g_i,i\in\mathcal{M}$. The goal is to optimize the sum of these functions for a common plan $z$: 
$\min_z \sum_{i\in\mathcal{M}} g_i(z)$. In order to leverage the separability of the agents, the problem is re-written to endow each agent with their own plan variable, subject to the constraint that they must all be equal:
\begin{align*}
\min_{\{x_i\}_{i\in\mathcal{M}},z}&~\sum_{i\in\mathcal{M}} g_i(x_i)\\
s.t.&~x_i=z,~\forall i \in \mathcal{M}.\nonumber
\end{align*}
Note that each agent's objective function could itself be the result of an optimization problem of the form $g_i(x):=\min_{y} g_i(x,y)$, where $y$ corresponds to a set of variables that are ``private'' to the agent.  
This formulation can be interpreted as the agents working towards agreeing on a common plan $z$, and optimizing their own objective function, given the consensus plan. 
Our objective is to define a generic Consensus Planning Protocol (CPP) to solve this general consensus problem. 
What algorithm we use depends in part on the interface the agents offer, in other words, what information they can consume, and what information they return.  
In practical systems, there are many ways agents can output information and interact with other agents and/or a coordinating algorithm.
We detail in the next section the three main types of agents considered in our consensus planning problem: primal, dual, and proximal agents.

\subsection{Agents}

In our consensus planning problem, there exist three types of agent interfaces: primal, dual, and proximal. These are the natural interfaces for the three main approaches to the problem: gradient descent, dual ascent, and the augmented Lagrangian method of multipliers, respectively.
\begin{description}
\item[Primal Agents:] A primal interface is the one that would be required in a first-order primal method such as gradient descent, and corresponds to the basic formulation of the consensus problem as:
\begin{align*}
\min_{z} \sum_{i\in\mathcal{M}} g_i(z).
\end{align*}
In a simple gradient descent scheme, we would in each iteration $k$ query the agents with the current tentative plan $z^k$, and they would return the gradients $\nabla g_i(z^k)$ at that point. We would then update the plan as $z^{k+1}=z^k -\rho \sum_{i\in\mathcal{M}} \nabla g_i(z^k)$, for some step size $\rho>0$. 
\item[Dual Agents:] A dual interface arises from a dualization of the constraint, and then solving of the dual problem as:
\begin{align*}
\max_{\{\lambda_i\}_{i\in\mathcal{M}}}\left\{\min_{\{x_i\}_{i\in\mathcal{M}}} \sum_{i\in\mathcal{M}} g_i(x_i) - \lambda_i^T (z-x_i)\right\}.
\end{align*}
We then require from the agents that they be able to consume the dual variables $\lambda_i$ and return the solutions to the subproblems in the brackets as $x_i=\arg\min_x  g_i(x) - \lambda_i^T (z-x)$. Note that the dual agent only requires the dual variables since the optimization is over $x$ and the term $\lambda_i^T z$ is a constant for that problem, which can be omitted.
\item[Proximal Agents:] A proximal interface results from the augmented Lagrangian formulation of the consensus problem, as:
\begin{align*}
\max_{\{\lambda_i\}_{i\in\mathcal{M}}}\left\{\min_{\{x_i\}_{i\in\mathcal{M}}} \sum_{i\in\mathcal{M}} g_i(x_i) - \lambda_i^T (z-x_i) + \frac{\rho}{2}\lVert z-x_i\rVert^2 \right\}.
\end{align*}
In this case, the agents need to be able to consume three pieces of information, prices $\lambda_i$, as in the dual agent, but also a consensus plan $z$, and a regularizing parameter $\rho$. They then return the solution to the subproblem: $x_i=\arg\min_x g_i(x) - \lambda_i^T (z-x) + \frac{\rho}{2}\lVert z-x\rVert^2$.
\end{description}

One of the motivations of the CPP is the ADMM algorithm and its favorable properties, which make the proximal interface a popular one. 
However, as a practical matter, in many realistic systems consisting of large complex agents, those agents cannot be readily ``turned into'' proximal agents.
Depending on their implementation, agents may naturally exist as primal or dual agents.
This requires that generic CPP algorithms be able to work across interface types.

\subsection{Previous Work}
 
Consensus planning is often viewed as a variant of distributed optimization, where there are relatively few optimization agents and querying information from each agent is comparatively costly.  
This variant is in contrast with much of the distributed optimization work where the central concern is that of scale across many agents \citep[ch.11]{ryu2022large}.  Given the requirement of a central coordinator, the setup is similar to that of federated learning.

Most of the distributed optimization algorithms assume the same interface for all agents: for example, a primal interface in the case of distributed (proximal) gradient descent; a dual interface in the case of dual decomposition; and a proximal interface in the case of distributed ADMM.  By contrast, this paper allows for any mix of primal, dual, and proximal interfaces. 

Heterogeneity of the agents is considered in the context of optimal control, however it usually takes a different meaning. \cite{sun2020distributed} for example consider coordination across agents that act in either discrete or continuous time.

A closer type of heterogeneity concerns the amount of information and computation that an agent can perform. \cite{niu2023fedhybrid} consider coordination among agents that can perform either gradient or Newton-type updates, i.e. first or second order agents. In our setting, both would still correspond to what we refer to as primal agents, and although our initial analysis focuses on the first-order formulation of the agents, and we address second order information in the extensions in Section~\ref{sec:second_order}.

\subsection{Motivating Examples}

We present in Table~\ref{tab:examples} several applications where consensus planning is relevant for a large online retailer.  
Note that, in many cases, the interface of a given agent is fixed by the existing implementation and design thereof, and thus must be accommodated by a generic CPP.  
These examples involve a few agents with varying types of interfaces, often times with mixed interfaces.  Part of the motivation of this work is to offer a single, unifying CPP algorithm to handle any such use case, in which different agents have different interfaces.

The different agent interfaces in Table~\ref{tab:examples} arise naturally from existing system architectures:
\begin{description}
\item[Dual interfaces] often correspond to agents that solve pricing-based subproblems (e.g., procurement systems that respond to price signals).
\item[Primal interfaces] correspond to agents that can evaluate costs and gradients but cannot easily solve optimization subproblems (e.g., simulation-based capacity models).
\item[Proximal interfaces] correspond to agents with flexible optimization capabilities (e.g., modern optimization-based planning systems).
\end{description}

This heterogeneity in interfaces reflects the reality of large-scale supply chain systems, where different components were developed at different times with different technological capabilities and constraints. Nonetheless, these systems are usually naturally convex. Buying agents for example typically stem from the modeling of a convex multi-period inventory management problem \citep{zipkin2000foundations}, and the dual interface allows for the problem to separate into product-level subproblems. Capacity agents also frequently involve convex costs and linear constraints, leading to convex problems, while transportation or flow agents are also naturally convex \citep{ahuja1988network}.

\begin{table}[t] %[h]
%\TABLE
\caption{Examples of consensus problems at a large online retailer.}\label{tab:examples}
\centering
\small
\begin{tabular}{p{3.5cm}p{10cm}}
\toprule
\textbf{Application} & \textbf{Description}\\
\midrule
Fullness  Optimization
& \emph{Objective}: coordinate inventory buying with physical network capacity. \\
& \emph{Agents}: buying agent and fullness (capacity) agent.  \\ 
& \emph{Interfaces}: dual for the buying agent; primal for the capacity agent. \\
\midrule
Throughput coordination
& \emph{Objective}: coordinate inventory flows throughout different regions (e.g., inbound, transfer, outbound),  given network labor constraints.   \\
& \emph{Agents}: each region. \\ 
& \emph{Interfaces}: proximal interface for all agents. \\
\midrule
Transportation optimization
& \emph{Objective}: coordinate transportation capacity across delivery stations and third party carriers. \\  
& \emph{Agents}: different stations, which cover overlapping geographical areas.  \\
& \emph{Interfaces}: proximal interface for all agents. \\ 
\midrule
Arrivals and throughput coordination
& \emph{Objective}: Like in throughput coordination,  coordinate inventory flows, but this differs because there is only one region and only one inventory flow (inbound arrivals) is considered. \\
& \emph{Agents}: buying agent (similar to fullness optimization problem) and throughput agent (analogous to throughput coordination problem). \\ 
& \emph{Interfaces}: buying agent has a dual interface; throughput agent may use a primal or dual interface. \\ 
\bottomrule
\end{tabular}
\end{table}
\section{Algorithm}

\subsection{Derivation}\label{sec:derivation}

We consider a consensus optimization problem involving a mixed set of agents $i\in\mathcal{M}$ that can be either primal, dual, or proximal. We let $\mathcal{P}$ be the set of primal agents, $\mathcal{D}$  be the set of dual agents, and $\mathcal{X}$ be the set of proximal agents. Each agent's objective function is given by $g_i$, where $i$ denotes their index, and we assume that the functions $g_i$ are convex, and in particular $\mu_i$-strongly convex for $i\in\mathcal{D}$, and have $\beta_i$-Lipschitz continuous gradients for $i\in\mathcal{X}$. The consensus problem to be solved is:
\begin{align}
\min_{\mathbf{x},z} &~\sum_{i \in \mathcal{M}} g_i(x_i)\label{eq:consensus_problem}\\
\text{s.t. }~& A z =\mathbf{x} \nonumber
\end{align}
where:
\begin{align*}
A&=\begin{bmatrix}
I\\
\vdots\\
I
\end{bmatrix},&
\mathbf{x}&=
\begin{bmatrix}
x_1\\
\vdots\\
x_M
\end{bmatrix}.
\end{align*}

We first make the different treatments of the agents explicit by separating agents based on their type. Applying partial duality (e.g., \cite[ch. 14]{Luenberger2021}) to the dual and proximal agents, and augmenting the latter with the augmented Lagrangian proximal term associated with their corresponding constraint, yields the following equivalent formulation:
\begin{align}
\max_{\lambda_i, i\in \mathcal{D}\cup \mathcal{X}} \min_{x_i,z} &~\sum_{i \in \mathcal{P}} g_i(x_i) + \sum_{i \in \mathcal{D}} g_i(x_i)+ \lambda_i^T (z-x_i) + \sum_{i \in \mathcal{X}} g_i(x_i)+ \lambda_i^T (z-x_i) + \frac{\rho_i}{2}\lVert z-x_i\rVert^2 ,
\label{eq:derivation}\\
\text{s.t.}~ & x_i=z\quad \forall i\in \mathcal{P}\nonumber
\end{align}
where we explicitly separate the three types of agents, and $\rho_i$ is the augmented Lagrangian parameter associated with agent $i\in\mathcal{X}$.  

Formulation~\eqref{eq:derivation} already suggests a direction for the solution of the problem,  owing in particular to the similarity in structure between ADMM \citep{boyd2011distributed} and traditional dual ascent methods \citep[ch. 14]{Luenberger2021} used to solve dualized consensus problems.
Both of these approaches involve iterations comprised of the same 3 steps:
\begin{enumerate}
\item $x_i$ update: update of the agents' primal variables,
\item $z$ update: update of the consensus plan through averaging,
\item $\lambda_i$ update: update of the dual variables.
\end{enumerate}
The similarity between dual and proximal agents can also be explained by the fact that one interpretation of the augmented Lagrangian relaxation is that it is simply the regular Lagrangian relaxation applied to a penalized version of the objective function (penalized by the addition of the quadratic term). That these two type of updates can be performed jointly is still to be proved, but it provides us with a framework within which to operate.  We then need to further incorporate the primal agents, which are queried with a tentative plan, $x_i$, and return gradient information at that point.  We detail in the next paragraph how we treat primal agents as approximate proximal agents, and we integrate them in the framework outlined above.

\paragraph{Primal agents as approximate proximal agents.}
A natural idea is to recast the primal agent as a proximal agent, where its function $g_i$ is replaced by some linearized approximation thereof using the gradient information returned by the primal agent. Fortunately, a substantive body of work has demonstrated that the ADMM updates do not need to be exact for the  method to converge \citep{he2002new,eckstein1992douglas,yang2011alternating}.
In particular, the terms of the ADMM updates can be linearized, be it the augmented quadratic term \citep{lin2011linearized,yang2013linearized}, or the agent function \citep{ouyang2013stochastic,suzuki2013dual}, or both through a more general use of a Bregman divergence term to replace the quadratic penalty, and an additional Bregman divergence term \citep{wang2014bregman}.  Leveraging these latter results, a simple Bregman-ADMM formulation in the context of the consensus problem has updates of the form:
\begin{align*}
x_i^{k+1} &= \arg\min_x g_i(x) - {\lambda_i^k}^T x + \frac{\rho_i}{2} \lVert z-x\rVert^2 + D_{\phi_i}(x,x_i^k),
\end{align*}
where $D_{\phi_i}$ is the Bregman distance associated with the convex function $\phi_i$ (see Appendix~\ref{def:bregman}). Letting $\phi_i$ be defined as:
\begin{align*}
\phi_i(x)&= \frac{L_i}{2}\lVert x \rVert^2 - g_i(x),
\end{align*}
where $L_i\geq \beta_i$ is a constant greater than the Lipschitz constant of $\nabla g_i$,  the update reads:
\begin{align}
x_i^{k+1} &= \arg\min_x g_i(x_i^k) + \nabla g_i(x_i^k)^T x + \frac{L_i}{2} \lVert x - x_i^k\rVert^2- {\lambda_i^k}^T x + \frac{\rho_i}{2} \lVert z^k-x\rVert^2.\label{eq:linearized_ADMM}
\end{align}
We can observe that the function $g_i$ has been linearized and replaced by a quadratic upper bound.  The solution to this subproblem is readily obtained as:
\begin{align*}
x_i^{k+1} &= \frac{L_i x_i^k + \rho_i z^k}{L_i+\rho_i} - \frac{1}{L_i+\rho_i}\left(\nabla g_i(x_i^k)-\lambda_i^k\right).
\end{align*}
We will see in Section~\ref{sec:comments} that when all agents are primal, the application of this linearized ADMM yields updates that bear a strong similarity to the ones performed by distributed gradient descent~\citep{nedic2009distributed}, and to the FedHybrid updates of \cite{niu2023fedhybrid}.

\subsection{Formulation}\label{sec:formulation}

We detail in this section the steps of the algorithm following the high level derivation presented in Section~\ref{sec:derivation}.  The agents can be either primal ($i\in\mathcal{P}$), dual ($i\in\mathcal{D}$), or proximal ($i\in\mathcal{X}$).  We recall that we consider $\mu_i$-strongly convex functions $g_i,i\in\mathcal{D}$ for dual agents, and $\beta_i$-Lipschitz continuous gradients for $g_i,i\in\mathcal{P}$.  We allow for $\mu_i=0$ and $L_i=+\infty$ for those other functions that are not strongly convex or do not have Lipschitz continuous gradients, for notational simplicity.  Each agent is allowed to have their own regularizing parameter/learning rate $\rho_i$, although in practice we often use the same one for agents of the same type. We will return to a discussion on the choice of these hyperparameters in Section~\ref{sec:enhancements}, but since the dual agents essentially perform dual ascent, and as we will confirm in the proof of the convergence results, we impose that $\rho_i<\mu_i,~\forall i\in\mathcal{D}$.  On the other hand, we recall that we require $L_i>\beta_i$ for $i\in\mathcal{P}$, and place no restriction on $\rho_i>0$ for $i\in \mathcal{X}$. We further assume throughout the paper that the initial prices $\lambda_i^0$ are such that $\sum_{i\in\mathcal{M}} \lambda_i=0$.

In each iteration of the algorithm, we alternate between agent updates, consensus update, and price updates as follows:
\begin{description}
\item[Agent Updates:]Update the agents $i\in\mathcal{M}$ in parallel,  with the following updates depending on the agent type.
\begin{description}
\item[Primal Agents ($i\in\mathcal{P}$):]
\begin{align}
x_i^{k+1} &= \arg\min_x g_i(x_i^k) + \nabla g_i(x_i^k)^T x + \frac{L_i}{2} \lVert x - x_i^k\rVert^2- {\lambda_i^k}^T x + \frac{\rho_i}{2} \lVert z^k-x\rVert^2,\nonumber\\
&= \frac{L_i x_i^k + \rho z^k}{L_i+\rho_i} - \frac{1}{L_i+\rho_i}\left(\nabla g_i(x_i^k)-\lambda_i^k\right).\label{eq:primal_update}
\end{align}
\item[Dual Agents ($i\in\mathcal{D}$):]
\begin{align}
x_{i}^{k+1} &= \arg\min_x g_i(x) - {\lambda_i^k}^T x . 
\label{eq:dual_update}
\end{align}
\item[Proximal Agents ($i\in\mathcal{X}$):]
\begin{align}
x_i^{k+1} &= \arg\min_x g_i(x) - {\lambda_i^k}^T x + \frac{\rho_i}{2} \lVert z^k-x\rVert^2 . 
\label{eq:proximal_update}
\end{align}
\end{description}
\item[Consensus Update:]
\begin{align*}
z^{k+1}&=\arg\min_z \sum_{i\in\mathcal{P}\cup\mathcal{D}\cup\mathcal{X}} {\lambda_i^k}^T z + \sum_{i\in\mathcal{P}\cup\mathcal{D}\cup\mathcal{X}} \frac{\rho_i}{2}  \lVert z - x_i^{k+1}\rVert^2  .
\end{align*}
The update takes the form of a weighted average of the agents' plans (using the fact that the sum of the price variables is null):
\begin{align}
z^{k+1}&= \frac{1}{\sum_{i\in\mathcal{M}}\rho_i} \sum_{i\in\mathcal{M}} \rho_i x_i^{k+1}.\label{eq:consensus_update}
\end{align}
\item[Price Updates:]
\begin{align}
\lambda_i^{k+1} &= \lambda_i^{k} + \rho_i \left(z^{k+1}-x_i^{k+1}\right) \qquad i \in \mathcal{P}\cup\mathcal{D}\cup\mathcal{X} . 
\label{eq:price_update}
\end{align}
\end{description}

\begin{remark}\label{rem:unification}
All the agents' problems can be expressed through the same equation as:
\begin{align}
x_i^{k+1} &= \arg\min_x g_i(x) - {\lambda_i^k}^T x + \frac{\tilde{\rho}_i}{2} \lVert z^k-x\rVert^2 + D_{\phi_i}(x,x_i^{k}),\label{eq:unified_update}
\end{align}
by letting:
\begin{align*}
\phi_i(x)&=\begin{cases}
0,&i\in\mathcal{X}\cup\mathcal{D}\\
\frac{L_i}{2}\lVert x\rVert^2 - g_i(x),& i\in\mathcal{\mathcal{P}}
\end{cases},&\tilde{\rho}_i&= 
\begin{cases}
\rho_i,&i\in\mathcal{\mathcal{P}\cup\mathcal{X}}\\
0,&i\in\mathcal{D}.
\end{cases}
\end{align*}
\end{remark}

Note that, technically, the update of the primal agents is performed by the coordinator. The primal agents are queried with the tentative plan $x_i^k$ and return $\nabla g_i(x_i^k)$, which is then used by the coordinator to yield the updated primal agent plan through \eqref{eq:primal_update}. Additionally, the consensus update should have a term comprised of a multiple of the sum of the prices, but summing up the price updates, \eqref{eq:price_update} we obtain the optimality condition of the consensus update problem,  meaning that the sum of the prices is null after the first iteration (see e.g. \cite{boyd2011distributed})\footnote{To avoid numerical errors preventing the sum of prices to be null in practice, a simple leveling step can be added after the price updates by subtracting the mean price from each price.}. Making the additional assumption that the sum of the initial prices is also null further simplifies the notation. The algorithm is summarized in Algorithm~\ref{alg:firstalgo}, which we name 3-Agent Consensus Planning (3ACP).

\begin{algorithm}[htbp]
\caption{Vanilla 3-Agent CP (3ACP)}\label{alg:firstalgo}
\begin{algorithmic}
\State Let $\rho_i>0,~\forall i\in\mathcal{M}$, and $\lambda_i^0$ such that $\sum_{i\in\mathcal{M}} \lambda_i^0=0$.
\State Initialize $x_i^0,~\forall i\in\mathcal{M}$ and let $z^0:=\frac{1}{\sum_{i\in\mathcal{M}}\rho_i} \sum_{i\in\mathcal{M}} \rho_i x_i^{0}$.
\While{convergence criterion not met}
\State Update the primal, dual, and proximal agents' plans $x_i^{k+1}$ using \eqref{eq:primal_update}, \eqref{eq:dual_update}, and \eqref{eq:proximal_update}, respectively.
\State Update the consensus plan $z^{k+1}$ using \eqref{eq:consensus_update}.
\State Update the agents' prices $\lambda_i^{k+1}$ using \eqref{eq:price_update}.
\EndWhile
\end{algorithmic}
\end{algorithm}

\subsection{Comments}\label{sec:comments}

We can glean some insight into the generic CPP algorithm by considering its behavior when all agents are of the same type.

\paragraph{Primal agents only.}
Consider the case when  all the agents are primal agents, and assume for simplicity that they use common values $L=L_i,\forall i$ and $\rho=\rho_i,\forall i$. 
The updates take the form of Bregman ADMM~\citep{wang2014bregman}:
\begin{align*}
x_i^{k+1}&=\frac{L x_i^k + \rho z^k}{L+\rho} - \frac{1}{L+\rho}\left(\nabla g_i(x_i^k)-\lambda_i^k\right),
\end{align*}
yielding consensus updates of the form:
\begin{align*}
z^{k+1} &= \frac{1}{|\mathcal{P}|}\sum_{i\in\mathcal{P}} x_i^{k+1}\\
&= z^k - \frac{1}{L+\rho} \sum_{i\in\mathcal{P}} \nabla g_i(x_i^k),
\end{align*}
by using the fact that the sum of the $\lambda_i$ is null.  

These updates bear strong similarities with a gradient descent type of update. In fact, letting:
\begin{align*}
\mathbf{x}^k &= \left[ x_1^k, x_2^k,\ldots, x_M^k\right]^T,\\
\nabla \mathbf{g}(x^k)&=\left[ \nabla g_1(x_1^k),\nabla g_2(x_2^k),\ldots,\nabla g_M(x_M^k)\right]^T,\\
\bm{\lambda}^k &= \left[ \lambda_1^k, \lambda_2^k,\ldots, \lambda_M^k\right]^T
\end{align*}
and
\begin{align*}
W&=
\frac{1}{L+\rho}\begin{bmatrix}
L + \frac{\rho}{M} & \frac{\rho}{M} & \ldots & \frac{\rho}{M} \\
 \frac{\rho}{M} & L+ \frac{\rho}{M} & \ldots & \frac{\rho}{M} \\
 \vdots & & \ddots & \vdots\\
  \frac{\rho}{M} &  \frac{\rho}{M} & \ldots & L+\frac{\rho}{M}
\end{bmatrix},
\end{align*}
the updates in the case of all primal agents can be rewritten as:
\begin{align*}
\mathbf{x}^{k+1}&= W \mathbf{x}^k - \alpha \left(\nabla \mathbf{g}(x^k)-\bm{\lambda}^k\right),
\end{align*}
with $\alpha = \frac{1}{L+\rho}$. 

Given that the matrix $W$ is symmetric and doubly stochastic, these updates are almost those of the \emph{decentralized gradient descent} \citep{nedic2009distributed}, except for the presence of dual variables $\lambda^k$ in the update.  Decentralized gradient descent requires that the step sizes be decreasing in order to converge to the optimal, otherwise it only converges to within a neighborhood of the optimal \citep{nedic2009distributed, yuan2016convergence}.  When the functions are strongly convex, the consensus plan converges linearly to the optimal until it reaches said neighborhood \citep{yuan2016convergence}. By working on both primal and dual variables through the coordinator, the linearized (Bregman) ADMM overcomes the need for decreasing step sizes and converges linearly to the optimal solution.

Another way of expressing the update of the preferred plans when all the agents are primal is as:
\begin{align*}
x_i^{k+1}&=x_i^k - \frac{1}{L+\rho} \left(\nabla g_i(x_i^k)-\lambda_i^k + \rho (x_i^k-z^k) \right),
\end{align*}
which is similar to the gradient-type update of the FedHybrid algorithm (\cite{niu2023fedhybrid}). A difference being that FedHybrid updates the dual variables using the previous consensus plan, while our approach based on the Bregman ADMM updates the consensus variable before updating the dual variables. Nonetheless, this suggests that FedHybrid could be interpreted as a particular case of Bregman ADMM.

\paragraph{Dual agents only.}
When all agents are dual and use common values $L=L_i,\forall i$, and $\rho=\rho_i,\forall i$, the algorithm yields updates of the form:
\begin{align*}
x_i^{k+1}&= \arg\min_{x_i} g(x_i)-{\lambda_i^k}^T x_i,\\
z^{k+1} &=\frac{1}{|\mathcal{D}|}\sum_{i\in\mathcal{D}} x_i^{k+1},\\
\lambda_i^{k+1} &= \lambda_i^{k} + \rho (z^{k+1}-x_i^{k+1}),
\end{align*}
 which are simply the updates resulting from a dual ascent procedure.

\paragraph{Proximal agents only.}
When all the agents are proximal, the algorithm reduces to the traditional consensus ADMM \citep{boyd2011distributed}.

\section{Convergence}

\subsection{Overview of Results}

We prove in this section different forms of convergence of Algorithm~\ref{alg:firstalgo} (3ACP).  
We saw in Section~\ref{sec:comments} that when the agents are exclusively primal, dual, or proximal, the algorithm reduces to Bregman ADMM, dual ascent, and regular ADMM, respectively. 
Individually, each of these algorithms is known to converge sublinearly when the functions are merely assumed to be generally convex, and strongly convex in the case of dual ascent. 
Under stricter conditions,  namely strong-convexity and Lipschitz continuous gradients, they reach linear convergence rates (see \cite{deng2016global,hong2017linear} for ADMM and \cite{wang2014bregman, lin2022alternating} for Bregman ADMM).  

Our main convergence result in Section~\ref{sec:sublinear_conv} below shows that when all three types of agents are combined, and with similar general convexity assumptions, strong convexity  for the dual agents, and Lipschitz-continuous gradients for primal agents, we preserve the sublinear convergence rate.  Furthermore, the additional strong convexity and Lipschitz continuity of the gradients for all agents allow us to establish two-step linear convergence, which we prove in Section~\ref{sec:linear_conv}.

We divide the convergence proofs in three parts: 
we first establish the plain convergence of the theorem in Section~\ref{sec:plain_conv};
we prove its sublinear convergence rate in Section~\ref{sec:sublinear_conv}; and 
we consider linear convergence in Section~\ref{sec:linear_conv}.

\subsection{Preliminaries and Assumptions}\label{sec:preliminaries}

We detail here the main assumptions and some additional notation used in the proofs. 

\begin{assumption}\label{ass:sublinear_assumptions}
Let $\mathcal{M}=\mathcal{P}\cup\mathcal{D}\cup\mathcal{X}$ be a set of primal ($\mathcal{P}$),  dual ($\mathcal{D}$), and proximal agents ($\mathcal{X}$), with objective functions $g_i,~i\in\mathcal{M}$.  For all $i\in\mathcal{M}$, the functions $g_i$ are convex, and in particular $\mu_i$-strongly convex for $i\in\mathcal{D}$, and with $\beta_i$-Lipschitz continuous gradients for $i\in\mathcal{P}$.
\end{assumption}

These assumptions are to be expected: the strong convexity of the dual agents is necessary for the convergence of dual ascent, which is a special case of 3ACP, while the Lipschitz-continuity of the primal agents is necessary in order to bound them above by a quadratic function. 

The second assumption has to do with the existence of a solution to the problem in the form of a saddle point for its (regular) Lagrangian, defined as:
\begin{align*}
\mathcal{L}(\mathbf{x},z,\bm{\lambda}):=\sum_{i\in\mathcal{M}} g_i(x_i) + \lambda_i^T (z-x_i).
\end{align*}

\begin{assumption}\label{ass:saddle_point}
The (regular) Lagrangian associated with the consensus problem \eqref{eq:consensus_problem} has a saddle point $(\mathbf{x}^*,z^*,\bm{\lambda}^*)$:
\begin{align}
\mathcal{L}(\mathbf{x}^*,z^*,\bm{\lambda})\leq \mathcal{L}(\mathbf{x}^*,z^*,\bm{\lambda}^*)\leq \mathcal{L}(\mathbf{x},z,\bm{\lambda}^*),\qquad\forall~\mathbf{x},z,\bm{\lambda}.\label{eq:saddle_point}
\end{align}
\end{assumption}

The final assumption concerns the learning rate of the dual agents, which we require to be less than the strong convexity parameter $\mu_i$, and is the usual assumption for dual ascent.

\begin{assumption}\label{ass:dual_step}
For $i\in\mathcal{D}$, the learning rate $\rho_i$ used in Algorithm~\ref{alg:firstalgo} is less than the strong-convexity bound of $g_i$: $\rho_i\leq \mu_i$.
\end{assumption}

\subsection{Plain Convergence}\label{sec:plain_conv}

We consider in this section the plain convergence of the algorithm. The first step is to leverage the existence of a saddle point $(\mathbf{x}^*,z^*,\bm{\lambda}^*)$ and the optimality conditions of the agents' subproblems to establish inequalities that will serve as the basis for the various convergence proofs.

\begin{proposition}\label{prop:main_ineq}
Consider the consensus problem \eqref{eq:consensus_problem} and suppose Assumptions~ \ref{ass:sublinear_assumptions} and \ref{ass:saddle_point}  hold. Then, the iterates generated by Algorithm~\ref{alg:firstalgo} satisfy the following inequality for all $k\geq 0$:
\begin{align}
0\leq\sum_{i\in\mathcal{M}} g_i(x_i^{k+1})-g_i(x_i^*) + {\lambda_i^*}^T(z^{k+1}-x_i^{k+1}) \leq V^k - V^{k+1} - r^{k+1},\label{eq:main_ineq}
\end{align}
leading to:
\begin{small}
\begin{align}
\sum_{i\in\mathcal{M}} \frac{\mu_i}{2} \lVert x_i^{k+1}-x_i^*\rVert^2  \leq V^k - V^{k+1} - r^{k+1},~\text{and }\sum_{i\in\mathcal{D}} \frac{\mu_i}{2L_i^2} \lVert \lambda_i^{k}-\lambda_i^*\rVert^2  \leq V^k - V^{k+1} - r^{k+1}. \label{eq:main_ineq_strong_conv}
\end{align}
\end{small}
\end{proposition}

\begin{proof}
See Appendix~\ref{app:proof_main_ineq}.
\end{proof}

The following proposition justifies the need for Assumption~\ref{ass:dual_step}, as it corresponds to the condition under which $r^{k+1}$ is non-negative.

\begin{proposition}\label{prop:pos_r}
Under Assumption~\ref{ass:dual_step}, we have:
\begin{align*}
r^{k+1}\geq 0,~\forall k\geq 0.
\end{align*}
\end{proposition}

\begin{proof}
See Appendix~\ref{app:pos_r}
\end{proof}

\begin{remark}\label{rem:lyapunov}
An important observation in Proposition~\ref{prop:main_ineq} is that under Assumption~\ref{ass:dual_step} we have in particular $V^{k+1} \leq V^k $, and thus that the (non-negative) sequence $\{V^k \}$ is non-increasing and thus bounded, and can serve as a \emph{Lyapunov} function.
\end{remark}

The next theorem concerns the convergence of the algorithm, its asymptotic primal feasibility, as well as convergence of the dual variables for the dual agents. When the set of dual agents is not null, Algorithm~\ref{alg:firstalgo} also implies convergence of the primal and dual variables of the dual agents, as well as convergence of the consensus variable.

\begin{theorem}[Convergence of Algorithm~\ref{alg:firstalgo}]\label{thm:convergence}
Consider the consensus problem \eqref{eq:consensus_problem} and suppose Assumptions~ \ref{ass:sublinear_assumptions},  \ref{ass:saddle_point} , and \ref{ass:dual_step} hold.  Then, for the iterates generated by Algorithm~\ref{alg:firstalgo} we have:
\begin{align*}
\sum_{i\in\mathcal{M}} g_i(x_i^{k+1}) - \sum_{i\in\mathcal{M}}  g_i(x_i^{*})& \to 0,\\
z^{k+1} - x_i^{k+1} \to 0, ~\forall i \in \mathcal{M}.
\end{align*}
Additionally, if $\mathcal{D}\neq \emptyset$, we have:
\begin{align*}
\lambda_i^k&\to\lambda_i^*,~\forall i \in \mathcal{D},\\
x_i^{k}&\to x_i^*,~\forall i \in\mathcal{M},\\
z^k &\to z^*.
\end{align*}
\end{theorem} 

\begin{proof}
See Appendix~\ref{app:proof_convergence}.
\end{proof}

\subsection{Sublinear Convergence}\label{sec:sublinear_conv}

We prove in this section the $O\left(\frac{1}{k}\right)$ ergodic convergence rate of the vanilla 3ACP algorithm \eqref{alg:firstalgo} under Assumptions~\ref{ass:sublinear_assumptions},  \ref{ass:saddle_point}, and \ref{ass:dual_step}.  By ergodic, we mean the convergence of the running average of the iterates.  While Theorem~\ref{thm:convergence} showed the plain convergence of Algorithm~\ref{alg:firstalgo}, we show here that at any iteration $K$, both the distance between the value of the the running average up to $K$, as well as its violation of the feasibility constraint decreases at a rate of $1/K$.

\begin{theorem}[Ergodic Sublinear Convergence]\label{thm:sublinear_conv}
Consider the consensus problem \eqref{eq:consensus_problem} and let Assumptions~\ref{ass:sublinear_assumptions}, \ref{ass:saddle_point}, and \ref{ass:dual_step} be satisfied. Let:
\begin{align*}
C&:=2 V^0, &\bar{\rho}&:=\max_i\{\rho_i\}, &\underline{\rho}&:=\min_i\{\rho_i\},
\end{align*}
and define the following running average iterates:
\begin{align*}
\mathbf{\hat{x}}^{K+1} &= \frac{1}{(K+1)}\sum_{k=0}^{K} \mathbf{x}^{k+1}, &\hat{z}^{K+1} &= \frac{1}{(K+1)}\sum_{k=0}^{K}z^{k+1}.
\end{align*}
Then, after $K$ iterations of Algorithm~\ref{alg:firstalgo}, we have:
\begin{align*}
\left\lvert \sum_{i\in\mathcal{M}} g_i(\hat{x}_i^{K+1}) - \sum_{i\in\mathcal{M}} g_i(x_i^*)\right\rvert &\leq \frac{C}{2(K+1)} + \frac{2 \sqrt{\bar{\rho} C}\lVert \bm{\lambda}^*\rVert}{\underline{\rho}(K+1)},\\
\left\lVert  A\hat{z}^{k+1}-\mathbf{\hat{x}}^{k+1} \right\rVert &\leq   \frac{2 \sqrt{\bar{\rho} C}}{\underline{\rho}(K+1)}.
\end{align*}
\end{theorem}

\begin{proof}
See Appendix~\ref{app:proof_sublinear_conv}.
\end{proof}

\subsection{Linear Convergence}\label{sec:linear_conv}

We prove in this section the two-step linear convergence of the 3ACP Algorithm. To do so, as well as to tighten the convergence bounds even in the sublinear case, we will assume that all agents have strongly convex objective functions with Lipschitz continuous~gradients, rather than the milder assumptions that only the dual agents be strongly convex, and only the primal agents have Lipschitz continuous gradients.  

\begin{assumption}\label{ass:linear_assumptions}
Let $\mathcal{M}$ be a set of primal,  dual,  and proximal agents, with objective functions $g_i,~i\in\mathcal{M}$.  For all $i\in\mathcal{M}$, the functions $g_i$ are  $\mu_i$-strongly convex with $\beta_i$-Lipschitz continuous~gradients.
\end{assumption}

We additionally define some terms and functions that will be useful in the proofs:
\begin{small}
\begin{align*}
V^k &:= \sum_{i\in\mathcal{M}} \frac{1}{2\rho_i}\lVert \lambda_i^k-\lambda_i^*\rVert^2 + \sum_{i\in\mathcal{P}\cup\mathcal{X}} \frac{\rho_i}{2} \lVert z^k - z^*\rVert^2 + \sum_{i\in\mathcal{P}} D_{\phi_i}(x_i^*, x_i^k)+ \sum_{i\in\mathcal{D}} D_{-g_i^*}(\lambda_i^k,\lambda_i^*),\\
r^k&:=  \sum_{i\in\mathcal{D}} \frac{1}{2\rho_i} \lVert \lambda_i^k-\lambda_i^{k-1}\rVert^2 - \sum_{i\in\mathcal{D}}D_{-g_i^*}(\lambda_i^k,\lambda_i^{k-1}) + \sum_{i\in\mathcal{P}\cup\mathcal{X}} \frac{\rho_i}{2} \lVert z^{k-1} - x_i^{k}\rVert^2 + \sum_{i\in\mathcal{P}} D_{\phi_i}(x_i^{k}, x_i^{k-1}) ,
\end{align*}
\end{small}
\noindent
where $g_i^*$ stands for the convex conjugate of $g_i$ (see Appendix~\ref{app:conv_conj}).  The function $V^k$ measures in some way the distance of the variables, both primal and dual, to their optimal values.  Enforcing Assumption~\ref{ass:dual_step} makes $r^k$ non-negative, and it is similar to the notion of residual in ADMM-related proofs, since $r^k$ being null would make the optimality conditions satisfied. 

\begin{remark}\label{rem:rk}
\begin{itemize}
\item When $\mathcal{D}=\emptyset$, $r^k$ can also be expressed as:
\begin{align*}
r^k&=  \sum_{i\in\mathcal{M}} \frac{1}{2\rho_i} \lVert \lambda_i^k-\lambda_i^{k-1}\rVert^2 + \sum_{i\in\mathcal{M}} \frac{\rho_i}{2} \lVert z^{k} - z^{k-1}\rVert^2 + \sum_{i\in\mathcal{P}} D_{\phi_i}(x_i^{k}, x_i^{k-1}),
\end{align*}
owing to the fact that $ \sum_{i\in\mathcal{M}}\lVert z^{k-1} - x^{k}\rVert^2 =  \sum_{i\in\mathcal{M}} \lVert z^{k-1} - z^k + z^k - x^{k}\rVert^2 =  \sum_{i\in\mathcal{M}}\lVert z^{k} - z^{k-1}\rVert^2 + \frac{1}{\rho_i^2} \lVert \lambda^{k} - \lambda^{k-1}\rVert^2 + \frac{2}{\rho_i} (z^{k-1} - z^k)^T (\lambda^k - \lambda^{k-1})= \sum_{i\in\mathcal{M}}\lVert z^{k} - z^{k-1}\rVert^2 + \frac{1}{\rho_i^2} \lVert \lambda^{k} - \lambda^{k-1}\rVert^2 $, where we used $\lambda_i^k - \lambda_i^{k-1}=\rho_i (z^k - x_i^k)$, and $\sum_{i\in\mathcal{M}}\lambda_i^k=0$.
\item When $\mathcal{M}=\mathcal{D}$,  we have:
\begin{align*}
V^k&= \sum_{i\in\mathcal{D}} D_{\psi_i}(\lambda_i^k, \lambda_i^*),\\
r^k&=  \sum_{i\in\mathcal{D}} D_{\psi_i}(\lambda_i^k, \lambda_i^{k-1}),
\end{align*}
where $\psi_i(\lambda):= -g_i^*(\lambda) + \frac{1}{2\rho_i}\|\lambda\|^2$.
\end{itemize}

\end{remark}

We then first define a few terms to simplify the notation, letting $\mathcal{S}\subset\mathcal{M}$ be any subset of agents:
\begin{comment}
\begin{align*}
\underline{\rho}&=\min_{i\in\mathcal{M}}\{\rho_i\},& \alpha&=\max_{i\in\mathcal{M}} \{L_i+L_{\phi_i}+\rho_i\},& L_{\phi_i}&=L_i-\mu_i,\\
\bar{\rho}&=\max_{i\in\mathcal{M}}\{\rho_i\},& \underline{\mu}&=\min_{i\in\mathcal{M}}\{\mu_i\},& \overline{L_\phi}&=\max_{i\in\mathcal{M}} \{L_{\phi_i}\}.
\end{align*}
\end{comment}
\begin{align*}
\underline{\rho}_\mathcal{S}&:=\min_{i\in\mathcal{S}}\{\rho_i\},& \alpha_\mathcal{S}&:=\max_{i\in\mathcal{S}} \{\beta_i+L_{\phi_i}+\rho_i\},& L_{\phi_i}&:=\beta_i-\mu_i,\\
\bar{\rho}_\mathcal{S}&:=\max_{i\in\mathcal{S}}\{\rho_i\},& \underline{\mu}_\mathcal{S}&:=\min_{i\in\mathcal{S}}\{\mu_i\},& \overline{L_\phi}&:=\max_{i\in\mathcal{P}} \{L_{\phi_i}\}.
\end{align*}

\begin{theorem}[Linear Convergence of 3ACP]\label{thm:linear_conv}
Consider the consensus problem \eqref{eq:consensus_problem} and let Assumptions~\ref{ass:linear_assumptions}, \ref{ass:saddle_point}, and \ref{ass:dual_step} be satisfied.  Then, the iterates generated by Algorithm~\ref{alg:firstalgo} satisfy the following two-step linear convergence rate:
\begin{align*}
 V^{k+2} \leq \left(1 + \frac{1}{2}\min\left\{\frac{\underline{\rho}_\mathcal{M}}{\alpha_\mathcal{M}}, \frac{\underline{\mu}_\mathcal{D}}{\alpha_\mathcal{D}} , \frac{\underline{\mu}_{\mathcal{P}\cup\mathcal{X}}}{\overline{\rho}_{\mathcal{P}\cup\mathcal{X}}},  \frac{\underline{\mu}_\mathcal{P}}{\overline{L_\phi}}\right\}\right)^{-1} V^{k}.
\end{align*}

\end{theorem}

\begin{proof}
See Appendix~\ref{app:proof_linear_conv}.
\end{proof}

We are only able to achieve a ``two-step'' linear convergence here (a linear relation between values two steps apart, $V^k$ and $V^{k+2}$), as opposed to the more common one-step linear convergence (relating consecutive values $V^k$ and $V^{k+1}$), which is the one found in the special cases where all the agents are of the same type. This can be understood by the slight shift in discretization implied by plain gradient schemes compared to proximal algorithms. Methods such as gradient descent can be interpreted as \emph{forward Euler methods}, while proximal algorithms have \emph{backward Euler method} interpretation \citep{parikh2014proximal}. To reconcile these two forms of approximation, we need to consider the progress after two steps of the algorithm.

Note that if all the agents are of the same type, we recover known results about the (one-step) linear convergence of the algorithm, corresponding to the linear convergence of ADMM in the case of all proximal agents (see \cite[Thm 3.4]{lin2022alternating}), of Bregman ADMM in the case of all primal agents (see \cite[Thm 3.8]{lin2022alternating}), and of dual ascent in the case of all dual agents. These results are summarized below.

\begin{corollary}[Linear Convergence of 3ACP in the case of a single interface]\label{cor:linear_conv}
Consider the consensus problem \eqref{eq:consensus_problem} and let Assumptions~\ref{ass:linear_assumptions}, \ref{ass:saddle_point}, and \ref{ass:dual_step} be satisfied.  Then, the iterates generated by Algorithm~\ref{alg:firstalgo} satisfy the following linear convergence rates when the agents are all of the same type:
\begin{description}
\item[All Primal Agents:]
\begin{align*}
 V^{k+1} \leq \left(1 + \frac{1}{3}\min\left\{\frac{\underline{\rho}_\mathcal{P}}{\alpha_\mathcal{P}} , \frac{\underline{\mu}_{\mathcal{P}}}{\overline{\rho}_{\mathcal{P}}},  \frac{\underline{\mu}_\mathcal{P}}{\overline{L_\phi}}\right\}\right)^{-1} V^{k}.
\end{align*}
\item[All Dual Agents:]
\begin{align*}
 V^{k+1} \leq \left(1 + \frac{1}{2}\min\left\{\frac{\underline{\rho}_\mathcal{D}}{\alpha_\mathcal{D}}, \frac{\underline{\mu}_\mathcal{D}}{\alpha_\mathcal{D}}\right\}\right)^{-1} V^{k}.
\end{align*}
\item[All Proximal Agents:]
\begin{align*}
 V^{k+1} \leq \left(1 + \frac{1}{2}\min\left\{\frac{\underline{\rho}_\mathcal{X}}{\alpha_\mathcal{X}},, \frac{\underline{\mu}_{\mathcal{X}}}{\overline{\rho}_{\mathcal{X}}}, \right\}\right)^{-1} V^{k}.
\end{align*}
\end{description}
\end{corollary}

\begin{proof}
See Appendix~\ref{app:proof_linear_conv_corollary}.
\end{proof}

\section{Practical Considerations}\label{sec:enhancements}

Algorithm~\ref{alg:firstalgo} presented the most basic version of the algorithm, and it can be improved upon in a number of different ways, either by relaxing some assumptions, or by incorporating algorithmic improvements. We consider in this section a few such extensions and practical considerations,  some of which could represent future research directions.

\subsection{Acceleration}\label{sec:acceleration}

The algorithms considered so far are essentially first-order methods, only making use of gradient information.  Such algorithms can often be accelerated through schemes that use previous iterates in to modify the direction of the step update, and use momentum. This is different from the ergodic convergence analysis of Section~\ref{sec:sublinear_conv}, where we considered the convergence of the running average of the iterates, but where that did not modify the step updates. One important such acceleration scheme is Nesterov's accelerated gradient descent \citep{nesterov1983method}, which has been successfully ported to many first order algorithms, including ADMM. This is especially interesting in our context since the primal and dual agents are handled as approximations of proximal agents, and are thus likely to lose some performance with respect to the~latter. 

Given that ADMM works both on the primal (individual plans and consensus) and dual (prices) variables,  acceleration techniques may target either of those sets of variables, or both.  For example, \cite{goldstein2014fast} presents an accelerated ADMM where the second set of primal variable and the prices are accelerated, while \cite{kadkhodaie2015accelerated} only modifies the dual variable but requires an additional update of the primal variables.  \cite{ouyang2013stochastic} considers accelerating the linearized ADMM algorithm used for the primal agents and is thus particularly relevant to us.  Additionally, accelerated methods are known to not be monotone in the objective value, and they usually display rippling effects.  Adaptive restart strategies to alleviate these issue were proposed in \cite{o2015adaptive}, and they can be adapted to the ADMM acceleration schemes, as was the case in \cite{goldstein2014fast}.

One difficulty is combining these different schemes. We can readily apply the accelerated scheme of \cite{goldstein2014fast} when the agents are dual and/or proximal, while we can also directly apply the algorithm of \cite{ouyang2013stochastic} when the agents are all primal, and some results resulting from those implementations are presented in Section~\ref{sec:quad_example}.  Combining these in order to allow for the acceleration of the algorithm for any combination of agents will be an interesting research topic.

\subsection{Tighter Quadratic Bounds for Primal Agents}\label{sec:quad_bounds}

Our approach to  primal agents involved the use of a quadratic approximation of the objective functions $g_i$ of the form $x\mapsto g_i(x_i^k) + \nabla g_i(x_i^k)^T (x-x_i^k) + \frac{L_i}{2}\lVert x\rVert^2$. This nonetheless restricts the type of linearization to spherical quadratic functions. We can tighten the bounds by considering not just spherical quadratic functions but more general quadratic functions so long as they dominate $g_i$. Letting $H_i$ be a symmetric positive definite matrix such that $H_i \succeq \nabla^2 g_i(x),~\forall x$, we may then use the following function $\phi_i$ in the definition of the Bregman divergence $D_{\phi_i}$: $\phi_i(x)=\frac{1}{2}x^T H_i x - g_i(x)=\frac{1}{2}\lVert x\rVert_{H_i}^2-g_i(x)$, instead of $\frac{L_i}{2}x^T x - g_i(x)$.

\subsection{Second Order Information}\label{sec:second_order}

Related to Section~\ref{sec:quad_bounds} above,  we might actually have access to second order information. This is especially useful for the primal and dual agents, since in the former case we could make use of a (close to) second-order approximation, while in the latter case, the price update could result from a Newton (or quasi-Newton) update as opposed to a simple first-order dual ascent. The use of second-order information in decentralized or federated learning has been considered in works such as \cite{zhang2015disco} or \cite{wang2018giant} for example.

For example,  for primal agents, we could consider a variable Bregman divergence $D_{\phi_i^k}$ of the form $\phi_i^k(x)=\frac{1}{2a_i}x^T(\nabla^2 g_i(x_i^k) +\epsilon I)  x - g_i(x)$, where $\varepsilon>0$ is used to ensure strong convexity. If we assume that the functions $g_i$ are strongly convex, this would yield updates similar to the Newton-type updates of the FedHybrid algorithm (\cite{niu2023fedhybrid}), reinforcing the view of FedHybrid as a particular case of Bregman ADMM.

For dual agents, we could consider a price update of the form $\lambda_i^{k+1} = \lambda_i^k + \rho_i \nabla^2 g_i(x_i^{k+1}) (z^{k+1}-x_i^{k+1})$,  or alternatively  $\lambda_i^{k+1} = \lambda_i^k + \rho_i (\nabla^2 g_i(x_i^{k+1}) + \epsilon I) (z^{k+1}-x_i^{k+1})$, for some $\epsilon>0$, where $\rho_i$ and $\epsilon$ are appropriately chosen to guarantee convergence, in particular so that $||\rho_i (\nabla^2 g_i(x_i^{k+1}) + \epsilon I)||<\mu_i$. In that case, however, the consensus update would need to be modified to:
\begin{align*}
z^{k+1}=\left(\sum_{i\in\mathcal{M}} H_i^{k+1}\right)^{-1} \left(\sum_{i\in\mathcal{M}} H_i^{k+1} x_i^{k+1}\right),
\end{align*}
where:
\begin{align*}
H_i^{k+1}:= \begin{cases}
\rho_i I,&\text{for }i\in\mathcal{P}\cup\mathcal{X},\\
\rho_i \nabla^2g_i(x_i^{k+1}),&\text{for }i\in\mathcal{D}
\end{cases}.
\end{align*}

\subsection{Choice of Hyperparameters}\label{sec:hyperparams}

One critical aspect in the implementation of Algorithm~\ref{alg:firstalgo},  similarly to ADMM algorithms,  is the choice of the hyperparameters $\rho_i,i\in\mathcal{M}$.  In ADMM, it is traditional to use a single parameter across agents, although there is a benefit to allowing for individual values, especially because different agents might have different scales of variables for which a single regularizing parameter might not be appropriate.  

Convergence bounds and in particular bounds characterizing the linear convergence of ADMM are often optimized for choices of $\rho_i$ given by the geometric mean of the strong convexity and gradient Lipschitz continuity: $\rho_i = \sqrt{\mu_i \beta_i}$, although this requires both that all the functions be strongly convex with Lipschitz continuous gradients, and that we have (approximate) knowledge thereof. On the other hand, dual ascent, which is essentially what the dual agents are performing, is optimized for $\rho_i=\mu_i$. As a result, even if we aim at using as few parameters as possible, it might be beneficial to set different values for dual agents than for primal/proximal ones.

In the absence of knowledge of $\mu_i$ and/or $\beta_i$, we can conceive of evaluating them ``on the fly'' based on accumulated information, or alternatively, dynamically adjusting them, as suggested by \cite{he2000alternating}. 
In particular, the self-adaptive parameter tuning can be applied at an agent level by considering their respective primal and and dual residuals.

We also note that the consensus update \eqref{eq:consensus_update} is performed as a weighted sum of the agents' preferred plans, where the weights are precisely given by the learning rates $\rho_i$.  Differences in the learning rates across agents imply differences in the relative importance of each agent's plan in computing the consensus plan.   The interaction between progress made by the individual agents' plans and their weight in the consensus update is not obvious.

\section{Example: Mixed Quadratic Agents}\label{sec:quad_example}

To illustrate the basic properties of our generic CPP, we consider in this section a synthetic example involving mixed quadratic agents. 
The advantage of such a setting, albeit simple, is that quadratic functions yield closed-form solutions that allow one to get more insight into how different agents are handled. 

\subsection{Agents}

Consider agents whose objective functions are given by the following, where $Q_i$ are symmetric definite positive matrices:
\begin{align*}
g_i(x) &= \frac{1}{2}x^T Q_i x + b_i^T x.
\end{align*}

\begin{description}
\item[Primal Agents:] When $i\in\mathcal{P}$, the agent receives $x_i^k$ and returns the gradient $\nabla g_i(x_i^k)=Q_ix_i^k +b_i$. Using \eqref{eq:primal_update}, we obtain an update of the form: 
\begin{align*}
x_i^{k+1}= \frac{(L_i I_n - Q_i)x_i^k + \rho_i z^k}{L_i+\rho_i} - \frac{b_i-\lambda_i^k}{L_i+\rho_i}.
\end{align*}
\item[Dual Agents:] When $i\in\mathcal{D}$, the agent receives $\lambda_i^k$ and returns $x_i^k$, the solution to the conjugate dual problem of $g_i$ at $\lambda_i^k$: 
\begin{align*}
x_i^{k+1}= Q_i^{-1}(\lambda_i^k-b_i).
\end{align*}
\item[Proximal Agents:] When $i\in\mathcal{X}$, the agent receives both $\lambda_i^k$ and $z^k$  and returns the solution to \eqref{eq:proximal_update}:
\begin{align*}
x_i^{k+1}= (Q_i+\rho_i I_n)^{-1} (\rho_i z^k + \lambda_i^k - b_i).
\end{align*}
\end{description}

\noindent
We rewrite these updates in a way that allows for a more direct comparison of their respective mechanisms.  Let $$\hat{x}_i^{k+1}:=\arg\min_x g_i(x)-{\lambda_i^k}^T x_i^k=Q_i^{-1}(\lambda_i^k-b_i)$$ be the solution to the dual problem, which in the case of a dual agent is simply the definition of $x_i^{k+1}$, but which we extend to the other types of agents using a different notation to avoid confusion. We then have:
\begin{small}
\begin{align*}
\text{primal update}\quad & x_i^{k+1} =& (L_i I_n +\rho_i I_n)^{-1} \rho_i z^k &+& (L_i I_n + \rho_i I_n)^{-1}L_i \hat{x}_i^{k+1}&+ \frac{(L_i I_n-Q_i)}{L_i +\rho_i} (x_i^k-\hat{x}_i^{k+1}) \\
\text{dual update} \quad& x_i^{k+1}=& &&  \hat{x}_i^{k+1}\\
\text{prox. update}\quad & x_i^{k+1}=& (Q_i +\rho_i I_n)^{-1} \rho_i z^k &+& (Q_i + \rho_i I_n)^{-1}Q_i \hat{x}_i^{k+1} .
\end{align*}
\end{small}
Formulated in this manner, the update equations yield a number of comments:

\begin{itemize}
\item All the updates share a common structure and interpretation as resulting from a weighted average of the current consensus plan $z^k$ and the tentative plan that would have resulted from the current price $\lambda_i^k$ in the absence of regularization, i.e., the solution resulting from the direct dualization of the constraint.
\item In the case of a spherical quadratic function of the form $\frac{L_i}{2}x^T x + b_i^T x$,  the updates of the primal and proximal agents would be exactly the same and would have the form:
\begin{align*}
x_i^{k+1}&= \frac{\rho_i z^k + L_i \hat{x}_i^{k+1}}{L_i+\rho_i}.
\end{align*}
%This would also be the exact same form of the updates for dual agents, were we to use the regularized dual update suggested in Section~\ref{sec:regularized_dual_update}.
\item The expressions  highlight the fact that $L_i I_n$ is a substitute for the Hessian of the agents' functions. This also emphasizes  the potential benefits of having tight quadratic bounds, as suggested in Section~\ref{sec:quad_bounds}. 
\item Compared to the proximal update, the primal update directly replaces the Hessian by $L_i I_n$. As a result, the weighing between the current consensus $z^k$ and $\hat{x}_i^{k+1}$ is the same in all directions, while in the proximal update, the effect of the Hessian $Q_i$ is to weigh different components differently towards one or the other vector.
\item The dual update doesn't just take a weighted average of $z^k$ and $\hat{x}_i^{k+1}$, it also contains an additional term that drags it towards the previous plan $x_i^k$.  The larger the gap between the upper quadratic bound and the function, the larger this effect. 
\item The interpretation of the update as a weighted average also provides some guidance as to how to set $\rho$. It is reasonable that we would want the weights to be of the same order of magnitude; and, as a result, we should set $\rho$ to have the same order of magnitude as the eigenvalues of the $Q_i$, for example, as the geometric mean of the smallest and largest eigenvalues $\sqrt{\mu \beta}$.
\end{itemize}

\subsection{Data}

We generate 30 random quadratic functions by generating symmetric definite matrices $Q_i$ as $Q_i = \alpha I_n + A_i^T A_i$, where $\alpha >0$ guarantees that the matrix is definite, and $A_i=r_{1i} (2 U_i - 1)$, where $r_{1i}>0$ and $U_i$ is a $n\times n$ matrix of random uniform numbers over [0,1].  We use $\alpha=1$ and $r_1=\tilde{U}_i$, which yields matrices with condition numbers ranging from a little over 1 to 60. We then generate 30 random vectors $b_i$ as  $b_i=r_2 u_i$, where $u_i$ is a random uniform vector over [0,1]. We used $r_2=1e4$.

The quadratic functions were then assigned to be:
1) all primal;
2) all dual;
3) all proximal;
4) one third primal, one third dual, and one third proximal;
5) one half primal and one half dual;
6) one half primal and one half proximal;
7) one half dual and one half proximal.

\subsection{Algorithm}

We solved all the configurations above using a version of the algorithm that allows for different regularizing parameters $\rho_p, \rho_d, \rho_x$, and also applied acceleration with adaptive restart in the configurations that allowed it (all primal,  all dual, all proximal, and part dual/proximal).

\subsection{Results}

The results of the algorithm on the different configurations are presented below for different levels of the regularizing parameters. We tried the following settings: 
\begin{itemize}
\item $\rho_p=\rho_d=\rho_x=0.1$ (Figure~\ref{fig:01} top left), 
\item $\rho_p=\rho_d=\rho_x=1$ (Figure~\ref{fig:01} top right), 
\item $\rho_p=\rho_x=10$ and $\rho_d=1$ (Figure~\ref{fig:01} bottom left),  
\item $\rho_p=\rho_x=50$ and $\rho_d=1$ (Figure~\ref{fig:01} bottom right). 
\end{itemize}
Here, $\rho_p$ is the learning rate used for primal agents, $\rho_d$ the learning rate used for dual agents, and $\rho_x$ the learning rate used for proximal agents.

The plots show the relative error of the objective function $f(z)=\sum_{i\in\mathcal{M}} g_i(z)$ for the consensus plan iterates $z_k$:  $\frac{f(z^k)-f(z^*)}{|f(z^*)|}$.

\begin{figure}[htbp]
\centering
\includegraphics[scale=0.35]{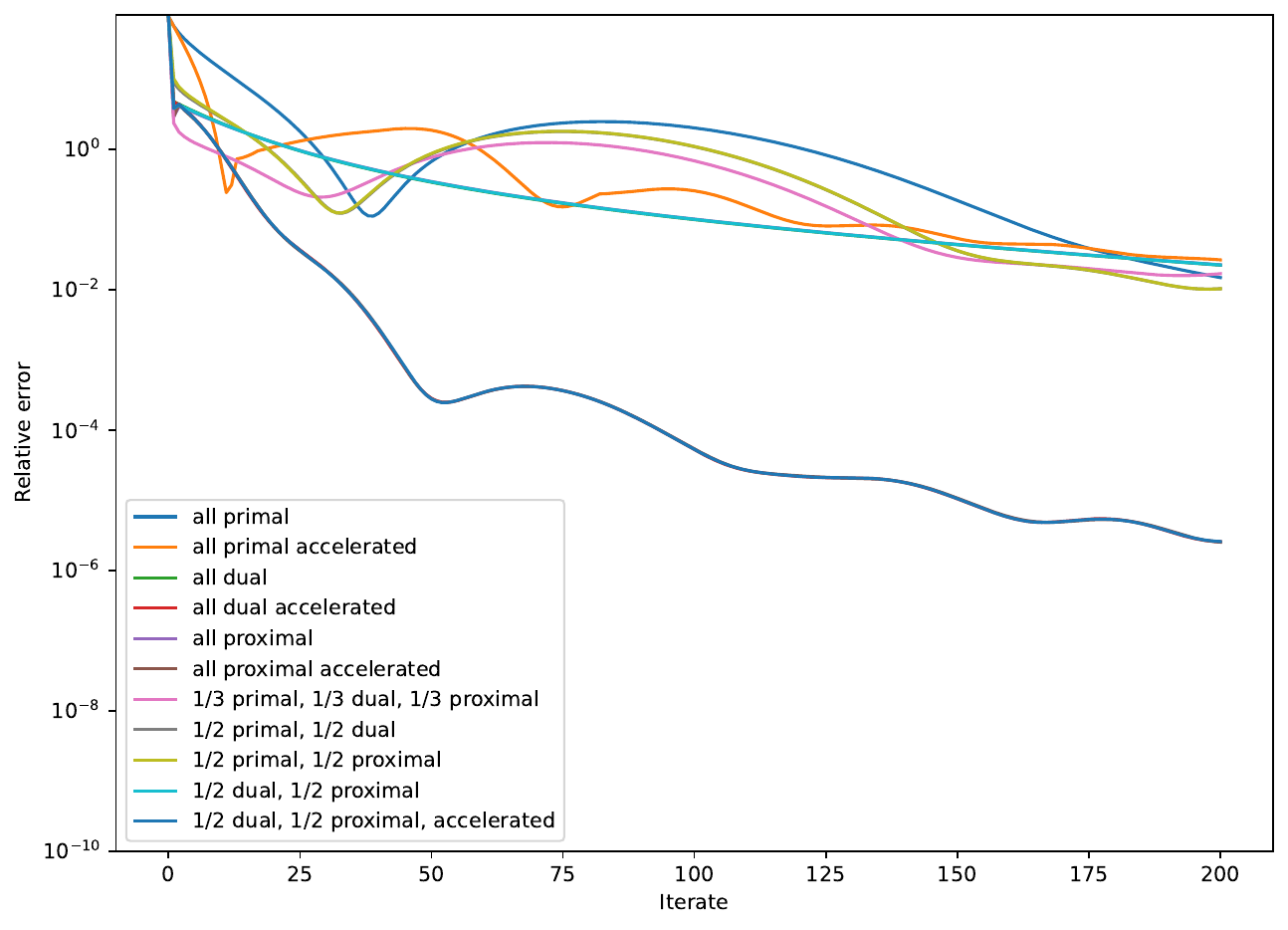}
\includegraphics[scale=0.35]{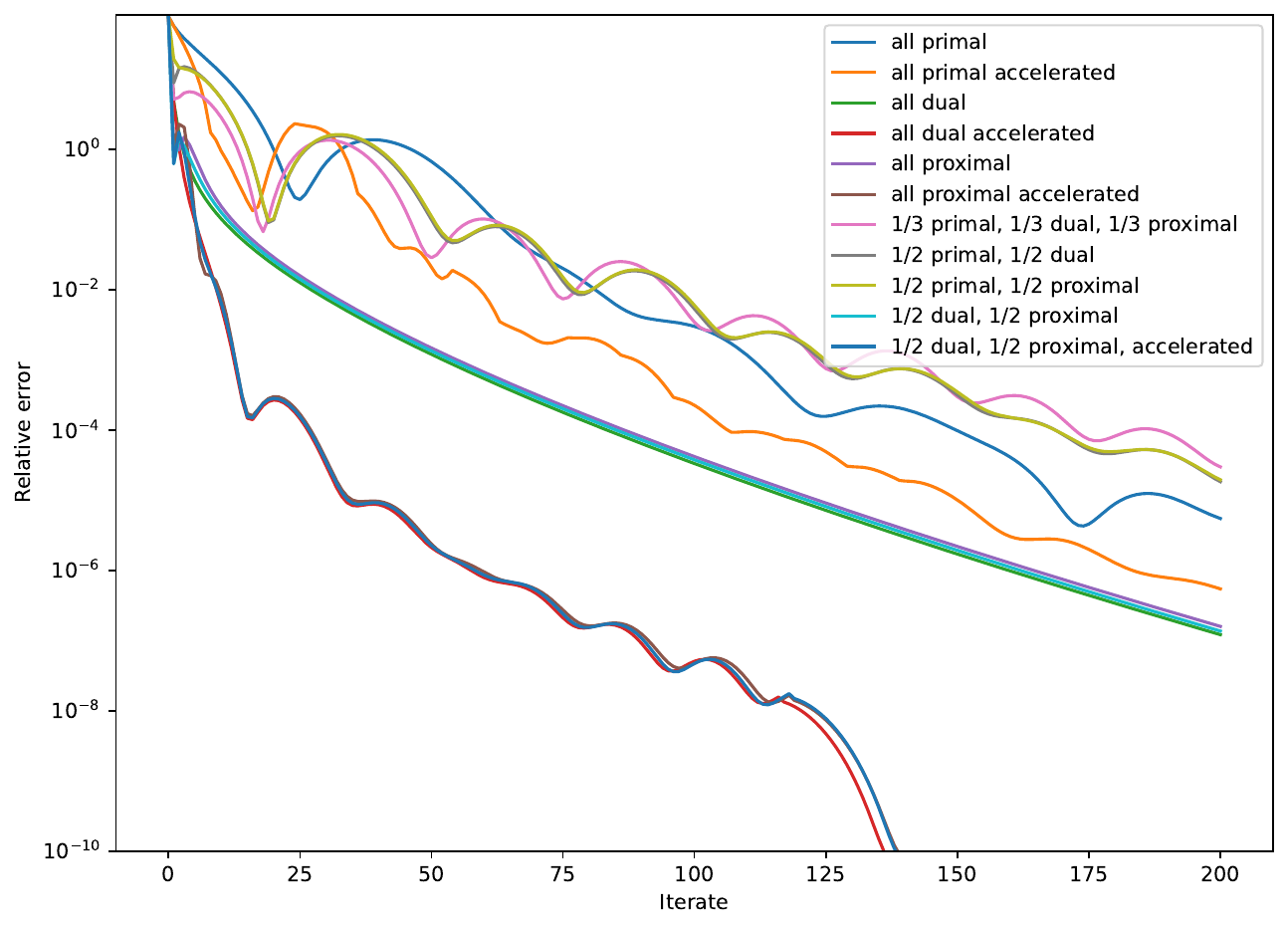}
\includegraphics[scale=0.35]{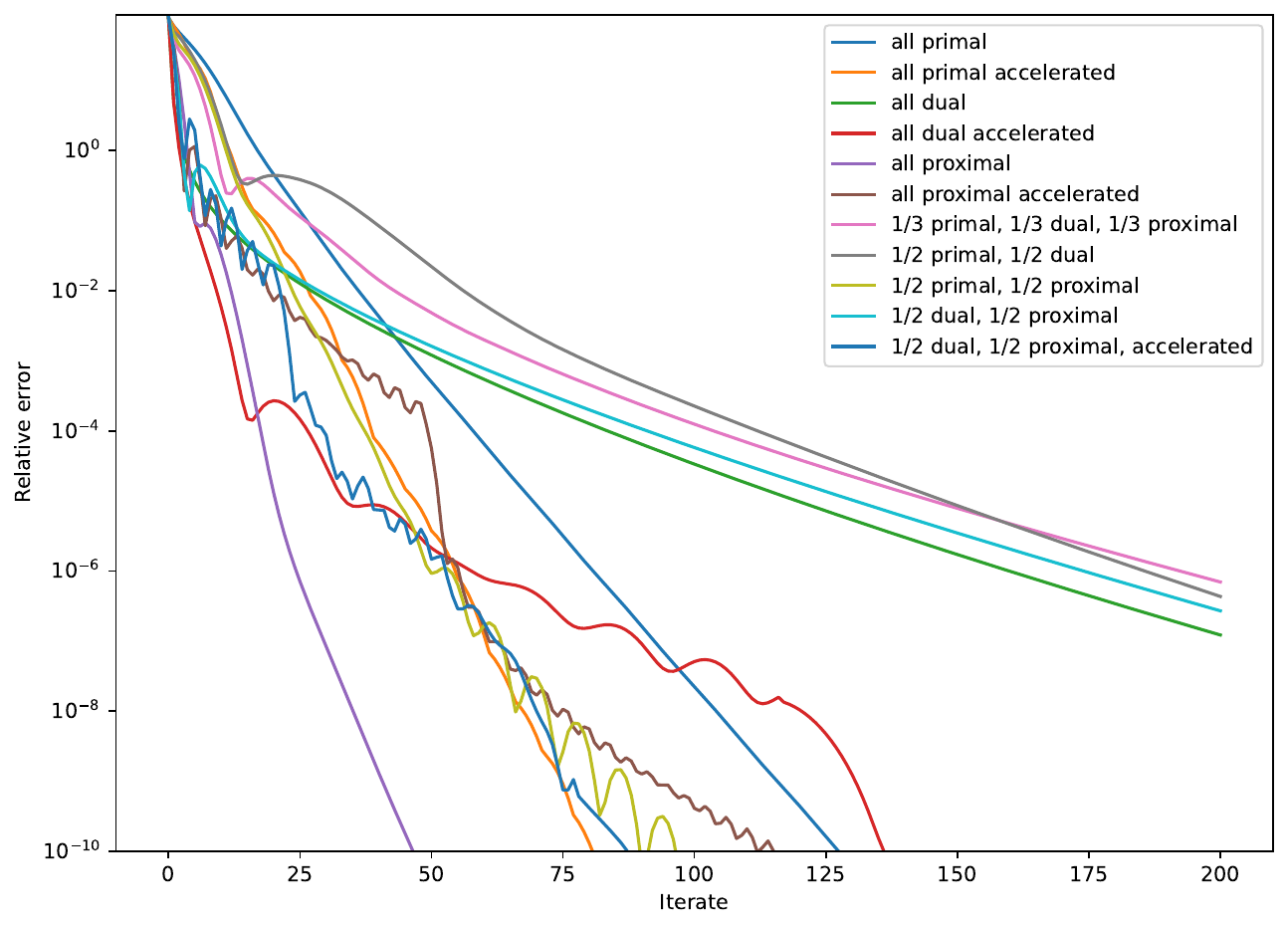}
\includegraphics[scale=0.35]{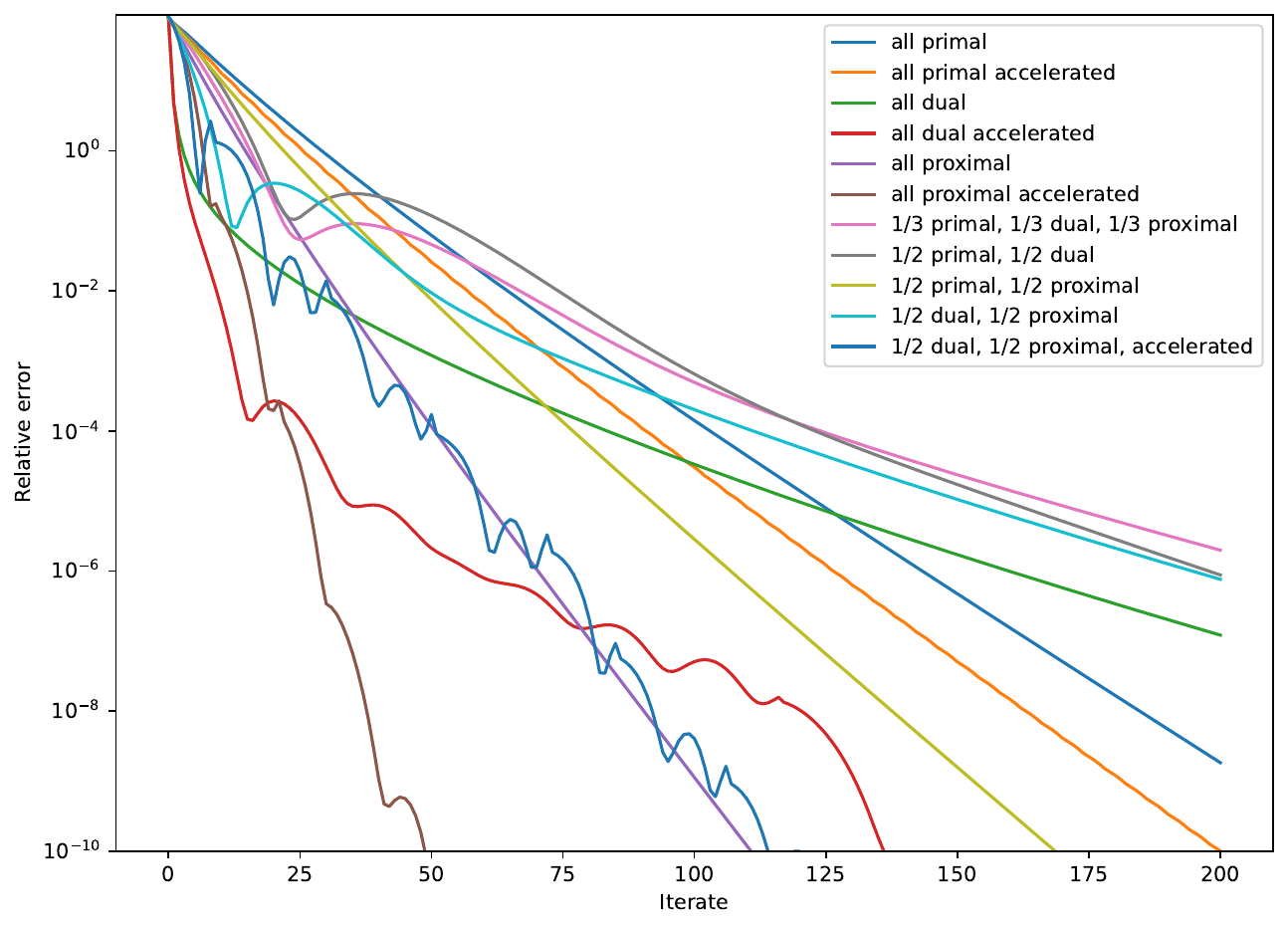}
\caption{Relative error of the consensus plan iterates for different configurations of the mix of agents and learning rates: $\rho_p=\rho_d=\rho_x=0.1$ (top left), $\rho_p=\rho_d=\rho_x=1$ (top right), $\rho_p=\rho_x=10$, $\rho_d=1$ (bottom left),  $\rho_p=\rho_x=50$, $\rho_d=1$ (bottom right).}\label{fig:01}
\end{figure}

The results in Figure~\ref{fig:01} highlight several points:
\begin{itemize}
\item Regardless of the mix of agents (which in many cases is constrained by the application and is not able to be adjusted), the general CPP achieves good convergence.
In each case, the convergence speed of the algorithm is impacted by the choice of the learning rate, which should thus be addressed in any practical implementation.
\item Overall, there is (if possible) a preference for proximal agents because of their superior theoretical properties and guarantees is reflected in the results.
\item The choice between dual and primal agents (when such a choice is possible) is not straightforward. Dual agents require knowledge or a good guess of the strong convexity constant, while proximal agents require the use of a Lipschitz constant. Nonetheless, dual agents require a single value to be set (the learning rate), while primal agents call for two (learning rate and upper bound on the Lipschitz constant), which could make the dual interface an easier one to implement.
\item Acceleration usually yields improvements, but also requires that some parameters be set.
\end{itemize}
\section{Conclusion}

We presented in this paper a general Consensus Planning Protocol that doesn't place any restriction on agent type, and thus allows for any combination of primal, dual, and proximal agents. We proved the convergence of the algorithm, and its sublinear $O(\frac{1}{k})$ convergence rate, as well as two-step linear convergence under strong convexity and Lipschitz continuous gradients for all functions. Numerical examples illustrate the behavior of the algorithm under different combinations of agent types, and the impact of acceleration. While we touched on some practical considerations, several of them will be explored in future work, in particular as relates to acceleration and asynchronous updates.  
\section{Acknowledgments}

This paper is part of  a stream of work to develop a market-based approach to coordination, the Consensus Planning Protocol (CPP) \citep{cpp,cpptalk}. We would especially like to thank Garrett van Ryzin for his support, as well as Manan Chopra, Shekhar Jain, Mitchell Joblin,  and Allam Reddy for their feedback.

\bibliographystyle{abbrvnat}
\bibliography{bib}

\appendix

\section{Proofs}

\subsection{Proof of Proposition~\ref{prop:main_ineq}}\label{app:proof_main_ineq}

The proof of the inequality makes use of the optimality conditions of both the agents' subproblems,  and the optimality condition of the main problem through the saddle point inequality \eqref{eq:saddle_point}. We thus start by stating the subproblems' optimality conditions, which we will subsequently use to transform the saddle point inequality and express it through relations that involve distances of the iterates to their optimal values. We will also use the following definition:
\begin{align*}
\tilde{V}^k &:= \sum_{i\in\mathcal{M}} \frac{1}{2\rho_i}\lVert \lambda_i^k-\lambda_i^*\rVert^2 + \sum_{i\in\mathcal{P}\cup\mathcal{X}} \frac{\rho_i}{2} \lVert z^k - z^*\rVert^2 + \sum_{i\in\mathcal{P}} D_{\phi_i}(x_i^*, x_i^k).
\end{align*}
\begin{description}
\item[Subproblems optimality conditions:]
Algorithm~\ref{alg:firstalgo} carries out two main types of subproblems, the agent updates, and the consensus update.  The agents' subproblems correspond to Equations~\eqref{eq:primal_update}, \eqref{eq:dual_update}, and \eqref{eq:proximal_update}. Their optimality conditions read:
\begin{align}
\nabla g_i(x_i^{k+1})&=  \lambda_i^{k} + \rho_i (z^{k}-x_i^{k+1}) - \left(\nabla \phi_i(x_i^{k+1}) - \nabla \phi_i(x_i^{k})\right),~i\in\mathcal{P}\label{eq:primal_opt_cond}\\
\nabla g_i(x_i^{k+1})&= \lambda_i^k + \rho_i (z^{k}-x_i^{k+1}),~i\in\mathcal{X},\label{eq:prox_opt_cond}\\
\nabla g_i(x_i^{k+1})&= \lambda_i^k,~i\in\mathcal{D}.\label{eq:dual_opt_cond}
\end{align}

\item[Saddle-point inequality:]
We now apply the right-hand side of the saddle-point inequality \eqref{eq:saddle_point} to the iterate $(\mathbf{x}^{k+1},z^{k+1},\bm{\lambda}^{k+1})$,  thus yielding:
\begin{align}
0&\leq \mathcal{L}(\mathbf{x}^{k+1},z^{k+1},\bm{\lambda}^*) - \mathcal{L}(\mathbf{x}^*,z^*,\bm{\lambda}^*)\nonumber \\
&= \sum_{i\in\mathcal{M}} g_i(x_i^{k+1}) - g_i(x_i^*)+ {\lambda_i^*}^T (z^{k+1}-x_i^{k+1}),\nonumber\\
&\leq \sum_{i\in\mathcal{M}} \nabla g_i(x_i^{k+1})^T (x_i^{k+1}-x_i^*)  + {\lambda_i^*}^T (z^{k+1}-x_i^{k+1}),\label{eq:sublinear_proof_dummy1}
\end{align}
using the convexity of the functions $g_i$. We then apply the optimality conditions \eqref{eq:primal_opt_cond}, \eqref{eq:prox_opt_cond}, and \eqref{eq:dual_opt_cond} to replace $\nabla g_i(x_i^{k+1})$ in the inequality above. For ease of exposition, we consider each term in turn.
\begin{itemize}
\item For the primal agents, we have:
\begin{align*}
&\sum_{i\in\mathcal{P}} \nabla g_i(x_i^{k+1})^T (x_i^{k+1}-x_i^*) +  {\lambda_i^*}^T (z^{k+1}-x_i^{k+1})\\
&= \sum_{i\in\mathcal{P}} \left( \lambda_i^{k} + \rho_i (z^{k}-x_i^{k+1}) - \left(\nabla \phi_i(x_i^{k+1}) - \nabla \phi_i(x_i^{k})\right)\right)^T \left(x_i^{k+1}-x_i^*\right)\\
&\qquad +  {\lambda_i^*}^T (z^{k+1}-x_i^{k+1})\\
&=\sum_{i\in\mathcal{P}} \left( \lambda_i^{k} + \rho_i (z^{k} - z^{k+1}) + \rho_i(z^{k+1}-x_i^{k+1}) - \left(\nabla \phi_i(x_i^{k+1}) - \nabla \phi_i(x_i^{k})\right)\right)^T \left(x_i^{k+1}-x_i^*\right) \\
&\qquad +  {\lambda_i^*}^T (z^{k+1}-x_i^{k+1})\\
&=\sum_{i\in\mathcal{P}} {\lambda_i^{k+1}}^T(x_i^{k+1}-z^{k+1}) + {\lambda_i^{k+1}}^T (z^{k+1}-z^*) + \rho_i (z^k-z^{k+1})^T(x_i^{k+1}-x_i^*)\\
&\qquad+  {\lambda_i^*}^T (z^{k+1}-x_i^{k+1})- \left(\nabla \phi_i(x_i^{k+1}) - \nabla \phi_i(x_i^{k})\right)^T \left(x_i^{k+1}-x_i^*\right) \\
 &=\sum_{i\in\mathcal{P}} -({\lambda_i^{k+1}}-\lambda_i^*)^T(x_i^{k+1}-z^{k+1}) + \rho_i (z^k-z^{k+1})^T(x_i^{k+1}-x_i^*) \\
 &\qquad - \left(\nabla \phi_i(x_i^{k+1}) - \nabla \phi_i(x_i^{k})\right)^T \left(x_i^{k+1}-x_i^*\right) + {\lambda_i^{k+1}}^T (z^{k+1}-z^*)\\
 &=\sum_{i\in\mathcal{X}} -\frac{1}{\rho_i}({\lambda_i^{k+1}}-\lambda_i^*)^T(\lambda_i^{k+1}-\lambda_i^{k}) + \rho_i (z^k-z^{k+1})^T(x_i^{k+1}-z^{k+1}) \\
 &\qquad +\rho_i (z^k-z^{k+1})^T (z^{k+1} -z^*)- \left(\nabla \phi_i(x_i^{k+1}) - \nabla \phi_i(x_i^{k})\right)^T \left(x_i^{k+1}-x_i^*\right)\\
 &\qquad + {\lambda_i^{k+1}}^T (z^{k+1}-z^*)\\
 &=\sum_{i\in\mathcal{P}} -\frac{1}{\rho_i}({\lambda_i^{k+1}}-\lambda_i^*)^T(\lambda_i^{k+1}-\lambda_i^{k}) - \rho_i (z^{k+1}-z^{k})^T (z^{k+1} -z^*) \\
 &\qquad + (z^{k+1}-z^{k})^T(\lambda_i^{k+1}-\lambda_i^{k})\\
 &\qquad- \left(\nabla \phi_i(x_i^{k+1}) - \nabla \phi_i(x_i^{k})\right)^T \left(x_i^{k+1}-x_i^*\right)+ {\lambda_i^{k+1}}^T (z^{k+1}-z^*)\\
 &=\sum_{i\in\mathcal{P}} -\frac{1}{2\rho_i}\lVert \lambda_i^{k+1}-\lambda_i^*\rVert^2 + \frac{1}{2\rho_i}\lVert \lambda_i^{k}-\lambda_i^*\rVert^2 - \frac{1}{2\rho_i}\lVert \lambda_i^{k+1}-\lambda_i^{k}\rVert^2\\
 &\qquad - \frac{\rho_i}{2} \lVert z^{k+1}-z^*\rVert^2 + \frac{\rho_i}{2}  \lVert z^{k}-z^*\rVert^2 - \frac{\rho_i}{2} \lVert z^{k+1}-z^{k}\rVert^2  \\
 &\qquad + \frac{1}{2\rho_i}\lVert \lambda_i^{k+1}-\lambda_i^k \rVert^2 + \frac{\rho_i}{2}\lVert z^{k+1}-z^k\rVert^2 - \frac{\rho_i}{2}  \lVert z^k - x_i^{k+1}\rVert^2\\
 &\qquad - \left(\nabla \phi_i(x_i^{k+1}) - \nabla \phi_i(x_i^{k})\right)^T \left(x_i^{k+1}-x_i^*\right)+ {\lambda_i^{k+1}}^T (z^{k+1}-z^*)\\
&=  \sum_{i\in\mathcal{P}} -\frac{1}{2\rho_i}\lVert \lambda_i^{k+1}-\lambda_i^*\rVert^2 + \frac{1}{2\rho_i}\lVert \lambda_i^{k}-\lambda_i^*\rVert^2 - \frac{\rho_i}{2} \lVert z^{k+1}-z^*\rVert^2 + \frac{\rho_i}{2}  \lVert z^{k}-z^*\rVert^2 \\
&\qquad -\frac{\rho_i}{2} \lVert z^k - x_i^{k+1}\rVert^2 \\
 &\qquad- \left(\nabla \phi_i(x_i^{k+1}) - \nabla \phi_i(x_i^{k})\right)^T \left(x_i^{k+1}-x_i^*\right)+ {\lambda_i^{k+1}}^T (z^{k+1}-z^*)
\end{align*}
Using Lemma~\ref{lem:bregman}, we find:
\begin{align}
&\sum_{i\in\mathcal{P}} \nabla g_i(x_i^{k+1})^T (x_i^{k+1}-x_i^*) +  {\lambda_i^*}^T (z^{k+1}-x_i^{k+1})\nonumber\\
&\leq \sum_{i\in\mathcal{P}} -\frac{1}{2\rho_i}\lVert \lambda_i^{k+1}-\lambda_i^*\rVert^2 + \frac{1}{2\rho_i}\lVert \lambda_i^{k}-\lambda_i^*\rVert^2 - \frac{\rho_i}{2} \lVert z^{k+1}-z^*\rVert^2 + \frac{\rho_i}{2}  \lVert z^{k}-z^*\rVert^2 \nonumber\\
&\qquad -\frac{\rho_i}{2} \lVert z^k - x_i^{k+1}\rVert^2 \nonumber\\
 &\qquad- D_{\phi_i}(x_i^*,x_i^{k+1}) +  D_{\phi_i}(x_i^*,x_i^{k}) -  D_{\phi_i}(x_i^{k+1},x_i^{k})\nonumber\\
&\qquad + {\lambda_i^{k+1}}^T (z^{k+1}-z^*). \label{eq:sublinear_proof_primal}
\end{align}
\item The sum over proximal agents is carried out identically to the primal agents, only with $\phi_i=0$, yielding:
\begin{align}
&\sum_{i\in\mathcal{X}} \nabla g_i(x_i^{k+1})^T (x_i^{k+1}-x_i^*) +  {\lambda_i^*}^T (z^{k+1}-x_i^{k+1})\nonumber\\
&\leq \sum_{i\in\mathcal{X}} -\frac{1}{2\rho_i}\lVert \lambda_i^{k+1}-\lambda_i^*\rVert^2 + \frac{1}{2\rho_i}\lVert \lambda_i^{k}-\lambda_i^*\rVert^2 - \frac{\rho_i}{2} \lVert z^{k+1}-z^*\rVert^2 + \frac{\rho_i}{2}  \lVert z^{k}-z^*\rVert^2 \nonumber\\
&\qquad -\frac{\rho_i}{2} \lVert z^k - x_i^{k+1}\rVert^2  + {\lambda_i^{k+1}}^T (z^{k+1}-z^*). \label{eq:sublinear_proof_prox}
\end{align}
\item For the dual agents, we have:
\begin{align}
&\sum_{i\in\mathcal{D}} \nabla g_i(x_i^{k+1})^T (x_i^{k+1}-x_i^*) +  {\lambda_i^*}^T (z^{k+1}-x_i^{k+1})\nonumber\\
&=\sum_{i\in\mathcal{D}}  {\lambda_i^k}^T (x_i^{k+1}-x_i^*) +  {\lambda_i^*}^T (z^{k+1}-x_i^{k+1})\nonumber\\
&=\sum_{i\in\mathcal{D}} {\lambda_i^k}^T (x_i^{k+1}-z^{k+1}) + \lambda_i^k (z^{k+1}-z^*) + {\lambda_i^*}^T (z^{k+1}-x_i^{k+1})\nonumber\\
&=\sum_{i\in\mathcal{D}} -\frac{1}{\rho_i} (\lambda_i^{k}-\lambda_i^*)^T (\lambda_i^{k+1}-\lambda_i^k) +  {\lambda_i^k}^T (z^{k+1}-z^*)\nonumber\\
&=\sum_{i\in\mathcal{D}} -\frac{1}{2\rho_i} \lVert \lambda_i^{k+1}-\lambda_i^*\rVert^2 + \frac{1}{2\rho_i} \lVert \lambda_i^{k}-\lambda_i^*\rVert^2 + \frac{1}{2\rho_i} \lVert\lambda_i^{k+1}-\lambda_i^k \rVert^2\nonumber\\
&\qquad + {\lambda_i^k}^T (z^{k+1}-z^*)\label{eq:sublinear_proof_dual}
\end{align}
\end{itemize}
Combining \eqref{eq:sublinear_proof_primal}, \eqref{eq:sublinear_proof_prox}, \eqref{eq:sublinear_proof_dual} into \eqref{eq:sublinear_proof_dummy1}, we find:
\begin{align}
0&\leq \mathcal{L}(\mathbf{x}^{k+1},z^{k+1},\bm{\lambda}^*) - \mathcal{L}(\mathbf{x}^*,z^*,\bm{\lambda}^*)\nonumber \\
&= \tilde{V}^k-\tilde{V}^{k+1}\nonumber \\
&\qquad -\sum_{i\in\mathcal{P}} D_{\phi_i}(x_i^{k+1},x_i^{k}) - \sum_{i\in\mathcal{P}\cup\mathcal{X}} \frac{\rho_i}{2} \lVert z^k - x_i^{k+1}\rVert^2  + \sum_{i\in\mathcal{D}} \frac{1}{2\rho_i} \lVert \lambda_i^{k+1} - \lambda_i^k\rVert^2\nonumber\\
& \qquad +\sum_{i\in\mathcal{P}\cup\mathcal{X}}   {\lambda_i^{k+1}}^T (z^{k+1}-z^*) + \sum_{i\in\mathcal{D}}   {\lambda_i^{k}}^T (z^{k+1}-z^*).\nonumber
\end{align}
We then recall that $\sum_{i\in\mathcal{M}} \lambda_i^k = 0$, so that $\sum_{i\in\mathcal{P}\cup\mathcal{X}} \lambda_i^k = - \sum_{i\in\mathcal{D}} \lambda_i^k$, leading to:
\begin{align}
0&\leq \mathcal{L}(\mathbf{x}^{k+1},z^{k+1},\bm{\lambda}^*) - \mathcal{L}(\mathbf{x}^*,z^*,\bm{\lambda}^*)\nonumber \\
&\leq \tilde{V}^k-\tilde{V}^{k+1}\nonumber \\
&\qquad -\sum_{i\in\mathcal{P}} D_{\phi_i}(x_i^{k+1},x_i^{k}) - \sum_{i\in\mathcal{P}\cup\mathcal{X}} \frac{\rho_i}{2} \lVert z^k - x_i^{k+1}\rVert^2  + \sum_{i\in\mathcal{D}} \frac{1}{2\rho_i} \lVert \lambda_i^{k+1} - \lambda_i^k\rVert^2\nonumber\\
& \qquad - \sum_{i\in\mathcal{D}}   \left({\lambda_i^{k+1}} - \lambda_i^k\right)^T (z^{k+1}-z^*) \nonumber\\
&=  \tilde{V}^k-\tilde{V}^{k+1}\nonumber \\
&\qquad -\sum_{i\in\mathcal{P}} D_{\phi_i}(x_i^{k+1},x_i^{k}) - \sum_{i\in\mathcal{P}\cup\mathcal{X}} \frac{\rho_i}{2} \lVert z^k - x_i^{k+1}\rVert^2  + \sum_{i\in\mathcal{D}} \frac{1}{2\rho_i} \lVert \lambda_i^{k+1} - \lambda_i^k\rVert^2\nonumber\\
& \qquad - \sum_{i\in\mathcal{D}}   \left({\lambda_i^{k+1}} - \lambda_i^k\right)^T (z^{k+1}-x_i^{k+1}) - \sum_{i\in\mathcal{D}}   \left({\lambda_i^{k+1}} - \lambda_i^k\right)^T (x_i^{k+1}-x_i^*)  \nonumber\\
&=  \tilde{V}^k-\tilde{V}^{k+1}\nonumber \\
&\qquad -\sum_{i\in\mathcal{P}} D_{\phi_i}(x_i^{k+1},x_i^{k}) - \sum_{i\in\mathcal{P}\cup\mathcal{X}} \frac{\rho_i}{2} \lVert z^k - x_i^{k+1}\rVert^2  + \sum_{i\in\mathcal{D}} \frac{1}{2\rho_i} \lVert \lambda_i^{k+1} - \lambda_i^k\rVert^2\nonumber\\
& \qquad - \sum_{i\in\mathcal{D}}   \frac{1}{\rho_i} \lVert \lambda_i^{k+1}-\lambda_i^k\rVert^2 - \sum_{i\in\mathcal{D}}   \left({\lambda_i^{k+1}} - \lambda_i^k\right)^T (x_i^{k+1}-x_i^*)  \nonumber\\
&=  \tilde{V}^k-\tilde{V}^{k+1}\nonumber \\
&\qquad -\sum_{i\in\mathcal{P}} D_{\phi_i}(x_i^{k+1},x_i^{k}) - \sum_{i\in\mathcal{P}\cup\mathcal{X}} \frac{\rho_i}{2} \lVert z^k - x_i^{k+1}\rVert^2  - \sum_{i\in\mathcal{D}} \frac{1}{2\rho_i} \lVert \lambda_i^{k+1} - \lambda_i^k\rVert^2\nonumber\\
& \qquad - \sum_{i\in\mathcal{D}}   \left({\lambda_i^{k+1}} - \lambda_i^k\right)^T (x_i^{k+1}-x_i^*). \label{eq:sublinear_proof_dummy2}
\end{align}
We focus on the last term in the sum,  and recall that for $i\in\mathcal{D}$,  $x_i^{k+1}=-\nabla g_i^*(\lambda_i^k)$ (see Proposition~\ref{prop:conv_conj}), whence:
\begin{align}
 & \sum_{i\in\mathcal{D}}-\left({\lambda_i^{k+1}} - \lambda_i^k\right)^T (x_i^{k+1}-x_i^*)\nonumber\\
 &= \sum_{i\in\mathcal{D}}- \left({\lambda_i^{k+1}} - \lambda_i^k\right)^T(\nabla -g_i^*(\lambda_i^{k})-\nabla -g_i^*(\lambda_i^*))\nonumber\\
 &= \sum_{i\in\mathcal{D}} D_{-g_i^*}(\lambda_i^k,\lambda_i^*) -  D_{-g_i^*}(\lambda_i^{k+1},\lambda_i^*) + D_{-g_i^*}(\lambda_i^{k+1},\lambda_i^k),\nonumber
\end{align}
using Lemma~\ref{lem:bregman}.
Substituting this last equality in \eqref{eq:sublinear_proof_dummy2}, we get:
\begin{align}
0&\leq \mathcal{L}(\mathbf{x}^{k+1},z^{k+1},\bm{\lambda}^*) - \mathcal{L}(\mathbf{x}^*,z^*,\bm{\lambda}^*) \leq  V^k - V^{k+1} -r^{k+1}, \nonumber
\end{align}
which is the first inequality in Proposition~\ref{prop:main_ineq}. 

To obtain the second inequality, we recall that $ \mathcal{L}(\mathbf{x}^{k+1},z^{k+1},\bm{\lambda}^*) - \mathcal{L}(\mathbf{x}^*,z^*,\bm{\lambda}^*) = \sum_{i\in\mathcal{M}} g_i(x_i^{k+1})-g_i(x_i^{*})+\lambda_i^*(z^{k+1}-x_i^{k+1})$. Using the convexity of the functions $g_i$, and in particular their strong convexity for $i\in\mathcal{D}$ (letting $\mu_i=0$ for those functions in $\mathcal{P}\cup\mathcal{X}$ that are not strongly convex), we then have:
\begin{align*}
&\sum_{i\in\mathcal{M}} g_i(x_i^{k+1})-g_i(x_i^{*})+\lambda_i^*(z^{k+1}-x_i^{k+1})\\
&= \sum_{i\in\mathcal{M}} g_i(x_i^{k+1})-g_i(x_i^{*})-\lambda_i^*(x_i^{k+1}-x_i^{*})+\lambda_i^*(z^{k+1}-z^{*})\\
&\geq \sum_{i\in\mathcal{M}} \frac{\mu_i}{2}\lVert x_i^{k+1} - x_i^* \rVert^2 + \sum_{i\in\mathcal{M}} \lambda_i^*(z^{k+1}-z^{*})\\
&=\sum_{i\in\mathcal{M}} \frac{\mu_i}{2}\lVert x_i^{k+1} - x_i^* \rVert^2 ,
\end{align*}
using the fact that $\sum_{i\in\mathcal{M}} \lambda_i^k =0$. This yields the second inequality in Proposition~\ref{prop:main_ineq}. 

The third one is then obtained by recalling that $x_i^{k+1}=-\nabla g_i^*(\lambda_i^k)$ for $i\in\mathcal{D}$,  and thus that $\lVert x_i^{k+1} - x_i^* \rVert^2=\lVert \nabla g_i^*(\lambda_i^k) - \nabla g_i^*(\lambda_i^k) \rVert^2$. Additionally,  since the functions $g_i,i\in\mathcal{D}$ have $L_i$-Lipschitz continuous gradients, their convex conjugates $g_i^*$ are $\frac{1}{L_i}$-strongly convex,  and so $\lVert \nabla g_i^*(\lambda_i^k) - \nabla g_i^*(\lambda_i^k) \rVert^2 \geq \frac{1}{L_i^2} \lVert \lambda_i^k - \lambda_i^*\rVert^2$, which gives the desired inequality.
\end{description}

\subsection{Proof of Proposition~\ref{prop:pos_r}}\label{app:pos_r}

To prove that $r^{k+1}\geq 0$, we need only prove that $\sum_{i\in\mathcal{D}} \frac{1}{2\rho_i} \lVert \lambda_i^k-\lambda_i^{k-1}\rVert^2 - \sum_{i\in\mathcal{D}}D_{-g_i^*}(\lambda_i^k,\lambda_i^{k-1})\geq 0$.  Since the functions  $g_i$  are $\mu_i$-strongly convex, the negative of their conjugates are convex with $\frac{1}{\mu_i}$-Lipschitz continuous gradients. Using Lemma~\ref{lem:bregman_lipschitz}, we have:
\begin{align*}
&\sum_{i\in\mathcal{D}} \frac{1}{2\rho_i} \lVert \lambda_i^k-\lambda_i^{k-1}\rVert^2 - \sum_{i\in\mathcal{D}}D_{-g_i^*}(\lambda_i^k,\lambda_i^{k-1})\\
&\geq \sum_{i\in\mathcal{D}}  \frac{1}{2\rho_i} \left(1 - \frac{\rho_i}{\mu_i}\right) \lVert \lambda_i^k-\lambda_i^{k-1}\rVert^2\\
&\geq 0,
\end{align*}
since by assumption we have $\frac{\rho_i}{\mu_i}\leq 1$.

\subsection{Proof of Theorem~\ref{thm:convergence}}\label{app:proof_convergence}

From Proposition~\ref{prop:main_ineq}, we have:
\begin{align*}
r^{k+1} \leq V^k -  V^{k+1}. 
\end{align*}
Summing this inequality over $k$ from 0 to $\infty$, and using the telescoping sum, we get:
\begin{align*}
\sum_{k=0}^{\infty} r^{k+1} &\leq V^0.
\end{align*}
Since the sequence $\{r^k\}$ is non-negative and its series sum is bounded above, it implies that $r^k\to 0$.  We also noted that the sequence $V^k$ is bounded, whence the sequences $\lVert \lambda_i^k - \lambda_i^*\rVert^2$, $\lVert z^k-z^*\rVert$, and $D_{\phi_i}(x_i^*,x_i^k)$ are also bounded.

We then consider the cases where $\mathcal{D}=\emptyset$ and $\mathcal{D}\neq\emptyset$ separately.
\begin{description}
\item[$\mathcal{D}=\emptyset$:]
When $\mathcal{D}=\emptyset$, and as shown in Section~\ref{sec:preliminaries}, we have:
\begin{align*}
r^k&=  \sum_{i\in\mathcal{M}} \frac{1}{2\rho_i} \lVert \lambda_i^k-\lambda_i^{k-1}\rVert^2 + \sum_{i\in\mathcal{M}} \frac{\rho_i}{2} \lVert z^{k} - z^{k-1}\rVert^2 + \sum_{i\in\mathcal{P}} D_{\phi_i}(x_i^{k}, x_i^{k-1}).
\end{align*}
It follows from the convergence of $\{r^k\}$ to 0 that $\lVert \lambda_i^k-\lambda_i^{k-1}\rVert^2 $,  $ \lVert z^{k} - z^{k-1}\rVert^2$, and $D_{\phi_i}(x_i^{k}, x_i^{k-1})$ also converge to 0 for all $i\in\mathcal{M}$.  Then, write $\lambda_i^{k+1}-\lambda_i^k=\rho_i (z^{k+1}-x_i^{k+1}) = \rho_i (z^{k+1}-z^*) - \rho_i(x_i^{k+1}-x_i^*)$. We established that $\lambda_i^{k+1}-\lambda_i^k\to 0$,  which implies that $z^{k+1}-x_i^{k+1}\to 0$ as well. 
We also wrote the bounded (and convergent) sequence $\lambda_i^{k+1}-\lambda_i^k$ as the sum of two sequences,  one of which is bounded,  implying that the other one, $x_i^k-x_i^*$, is bounded as well.

We now turn to the convergence of the objective function, and show that the difference $\sum_{i\in\mathcal{M}} g_i(x_i^{k+1}) - \sum_{i\in\mathcal{M}}  g_i(x_i^{*})$ can be bounded below and above by sequences that converge to 0.
\begin{itemize}
\item We first use the convexity of the functions to bound the difference above:
\begin{align*}
\sum_{i\in\mathcal{M}} g_i(x_i^{k+1}) - \sum_{i\in\mathcal{M}}  g_i(x_i^{*})&\leq \sum_{i\in\mathcal{M}}  \nabla g_i(x_i^{k+1})^T (x_i^{k+1}-x_i^*).
\end{align*}
The expressions for the gradients $\nabla g_i(x_i^{k+1})$ were given in \eqref{eq:primal_opt_cond}, \eqref{eq:primal_opt_cond}, and \eqref{eq:dual_opt_cond}, yielding:
\begin{align*}
&\sum_{i\in\mathcal{M}} g_i(x_i^{k+1}) - \sum_{i\in\mathcal{M}}  g_i(x_i^{*})\\
&\leq \sum_{i\in\mathcal{M}} \left(\lambda_i^k + \rho_i (z^{k}-x_i^{k+1})\right)^T (x_i^{k+1}-x_i^*) \\
&\qquad + \sum_{i\in\mathcal{P}} \left(\nabla \phi_i(x_i^{k+1})-\nabla \phi_i(x_i^k)\right)^T (x_i^{k+1}-x_i^*).
\end{align*}
The first sum can be transformed to:
\begin{align*}
&\sum_{i\in\mathcal{M}} \left(\lambda_i^k + \rho_i (z^{k}-x_i^{k+1})\right)^T (x_i^{k+1}-x_i^*)  \\
&=\sum_{i\in\mathcal{M}}   \left(\lambda_i^{k+1} + \rho_i (z^{k}-z^{k+1})\right)^T (x_i^{k+1}-x_i^*)\\
&=\sum_{i\in\mathcal{M}}   {\lambda_i^{k+1}}^T (x_i^{k+1}-x_i^*) + \rho_i (z^{k}-z^{k+1})^T(x_i^{k+1}-x_i^*)\\
&=\sum_{i\in\mathcal{M}}  -\frac{1}{\rho_i} {\lambda_i^{k+1}}^T(\lambda_i^{k+1}-\lambda_i^k) + {\lambda_i^{k+1}}^T (z^{k+1}-z^*) + \rho_i (z^{k}-z^{k+1})^T(x_i^{k+1}-x_i^*)\\
&=\sum_{i\in\mathcal{M}}  -\frac{1}{\rho_i} {\lambda_i^{k+1}}^T(\lambda_i^{k+1}-\lambda_i^k) + \rho_i (z^{k}-z^{k+1})^T(x_i^{k+1}-x_i^*).
\end{align*}
We thus have:
\begin{align}
&\sum_{i\in\mathcal{M}} g_i(x_i^{k+1}) - \sum_{i\in\mathcal{M}}  g_i(x_i^{*})\nonumber\\
&\leq \sum_{i\in\mathcal{M}}  -\frac{1}{\rho_i} {\lambda_i^{k+1}}^T(\lambda_i^{k+1}-\lambda_i^k) + \rho_i (z^{k}-z^{k+1})^T(x_i^{k+1}-x_i^*)\nonumber\\
&~\qquad + \sum_{i\in\mathcal{P}} \left(\nabla \phi_i(x_i^{k+1})-\nabla \phi_i(x_i^k)\right)^T (x_i^{k+1}-x_i^*)\nonumber\\
&\leq \sum_{i\in\mathcal{M}} \frac{1}{\rho_i} \lVert \lambda_i^{k+1}\rVert \lVert \lambda_i^{k+1}-\lambda_i^k\rVert + \rho_i \lVert z^{k}-z^{k+1}\rVert \lVert x_i^{k+1}-x_i^*\rVert \nonumber\\
&\qquad + \sum_{i\in\mathcal{P}} \lVert \nabla \phi_i(x_i^{k+1})-\nabla \phi_i(x_i^k)\rVert \lVert x_i^{k+1}-x_i^*\rVert.\label{eq:conv_proof_dummy}
\end{align}
We then recall that the sequence $\lVert \lambda_i^k\rVert$ is bounded while $\lVert \lambda_i^{k+1}-\lambda_i^k\rVert \to 0$; that $\lVert z^{k+1}-z^k\rVert \to 0$ while $\lVert x_i^{k+1}-x_i^*\rVert$ is bounded; and that since $D_{\phi_i}(x_i^{k+1},x_i^k)\to 0$, we have $x_i^k - x_i^{k+1} \to 0$, and as a result $ \lVert \nabla \phi_i(x_i^{k+1})-\nabla \phi_i(x_i^k)\rVert \to 0$ by continuity.  Put all together, this implies that \eqref{eq:conv_proof_dummy} converges to 0.
\item On the other hand, the first inequality from Proposition~\ref{prop:main_ineq} gives:
\begin{align*}
\sum_{i\in\mathcal{M}} g_i(x_i^{k+1}) - \sum_{i\in\mathcal{M}}  g_i(x_i^{*}) &\geq \sum_{i\in\mathcal{M}} - {\lambda_i^*}^T\left(z^{k+1}-x_i^{k+1}\right)\\
& \geq \sum_{i\in\mathcal{M}} - \lVert \lambda_i^* \rVert \lVert z^{k+1}-x_i^{k+1}\rVert\\
&\to 0,
\end{align*}
since we showed that $z^{k+1}-x_i^{k+1}\to 0$.
\end{itemize}

\item[$\mathcal{D}\neq\emptyset$:] When $\mathcal{D}\neq\emptyset$,  we have:
\begin{align*}
r^k&:= \sum_{i\in\mathcal{D}} \frac{1}{2\rho_i} \lVert \lambda_i^k-\lambda_i^{k-1}\rVert^2 - \sum_{i\in\mathcal{D}}D_{-g_i^*}(\lambda_i^k,\lambda_i^{k-1}) + \sum_{i\in\mathcal{P}\cup\mathcal{X}} \frac{\rho_i}{2} \lVert z^{k-1} - x_i^{k}\rVert^2\\
&\qquad + \sum_{i\in\mathcal{P}} D_{\phi_i}(x_i^{k}, x_i^{k-1}).
\end{align*}
It follows from the convergence of $\{r^k\}$ to 0 that $\lVert \lambda_i^k-\lambda_i^{k-1}\rVert^2$ converges to 0 for all $i\in\mathcal{D}$, as well as $D_{\phi_i}(x_i^{k+1},x_i^k)$ for all $i\in\mathcal{P}$, and $z^{k-1}-x_i^k\to 0$ for all $i\in\mathcal{P}\cup\mathcal{X}$.  Using Proposition~\ref{prop:main_ineq},  and summing the second and third inequalities from 0 to $\infty$ we also have:
\begin{align*}
\sum_{k=0}^{\infty}\sum_{i\in\mathcal{D}} \frac{\mu_i}{2} \lVert x_i^{k+1}-x_i^*\rVert^2  &\leq V^0 ,\\
\sum_{k=0}^{\infty}\sum_{i\in\mathcal{D}} \frac{\mu_i}{2L_i^2} \lVert \lambda_i^{k}-\lambda_i^*\rVert^2  &\leq V^0,
\end{align*}
whence $ \lVert x_i^{k+1}-x_i^*\rVert^2\to 0$ and $ \lVert \lambda_i^{k}-\lambda_i^*\rVert^2 \to 0$ for $i\in\mathcal{D}$, i.e. $x_i^k\to x_i^*$, and $\lambda_i^k\to \lambda_i^*$.

Then, since $\rho_i(z^{k+1}-z^*)=\lambda_i^{k+1} - \lambda_i^k + \rho_i (x_i^{k+1}-x_i^*)$ and both $\lambda_i^{k+1} - \lambda_i^k\to 0$ and $x_i^{k+1}-x_i^*\to 0$ for $i\in\mathcal{D}$, we also have $z^{k+1}-z^*\to 0$.  It follows immediately from the convergence of $x_i^k$ and $z^k$ to $x_i^*=z^*$ for $i\in\mathcal{D}$ that we also have $z^{k+1}-x_i^{k+1}\to 0$ for all $i\in\mathcal{D}$.  For $i\in\mathcal{P}\cup\mathcal{X}$, since we have both that $z^k\to z^*$ and $z^{k-1}-x_i^k\to 0$, it follows that $x_i\to z^*=x_i^*$.

For the proof of the convergence of the objective function, we proceed just as in the case where $\mathcal{D}=\emptyset$ by bounding the difference  $\sum_{i\in\mathcal{M}} g_i(x_i^{k+1}) - \sum_{i\in\mathcal{M}}  g_i(x_i^{*})$ by sequences that converge to~0.
\begin{itemize}
\item  Using the convexity of the functions $g_i$, we have:
\begin{align*}
\sum_{i\in\mathcal{M}} g_i(x_i^{k+1}) - \sum_{i\in\mathcal{M}}  g_i(x_i^{*})&\leq \sum_{i\in\mathcal{M}}  \nabla g_i(x_i^{k+1})^T (x_i^{k+1}-x_i^*).
\end{align*}
Using transformations similar to the ones performed above, we find:
\begin{align*}
&\sum_{i\in\mathcal{M}} g_i(x_i^{k+1}) - \sum_{i\in\mathcal{M}}  g_i(x_i^{*})\\
&\leq \sum_{i\in\mathcal{P}\cup\mathcal{X}} \rho_i \left(z^k-z^{k+1}\right)^T \left(x_i^{k+1}-x_i^*\right)  \\
&\qquad + \sum_{i\in\mathcal{P}} \left(\nabla \phi_i(x_i^{k+1})-\nabla \phi_i(x_i^k)\right)^T (x_i^{k+1}-x_i^*) \\
&\qquad + \sum_{i\in\mathcal{D}} \left(\lambda_i^{k} -\lambda_i^{k+1}\right)^T(x_i^{k+1}-x_i^*).
\end{align*}
The same arguments using the boundedness of the sequence $\lVert x_i^{k+1} -x_i^*\rVert$, and the convergences to 0 of $\lVert z^{k+1}-z^k\rVert$, $\lVert x_i^{k+1} - x_i^k\rVert$ for $i\in\mathcal{P}$,  and $\lVert \lambda_i^{k+1}-\lambda_i^k\rVert$ for $i\in\mathcal{D}$.
\item The convergence of the lower bound to 0 is identical to the case where $\mathcal{D}=\emptyset$.
\end{itemize}
\end{description}

\subsection{Proof of Theorem~\ref{thm:sublinear_conv}}\label{app:proof_sublinear_conv}

Consider the first inequality in Proposition~\ref{prop:main_ineq}, and average it over $k=0,\ldots,K$:
\begin{align*}
\frac{1}{K+1}\sum_{k=0}^K \left\{\sum_{i\in\mathcal{M}} g_i(x_i^{k+1}) - \sum_{i\in\mathcal{M}} g_i(x_i^*) \right\}+  {\bm{\lambda}^*}^T (A\hat{z}^{K+1}-\mathbf{\hat{x}}^{K+1}) \leq \frac{C}{2(K+1)}.
\end{align*}
Using the convexity of the functions $g_i$, this yields:
\begin{align*}
\sum_{i\in\mathcal{M}} g_i(\hat{x}_i^{K+1}) - \sum_{i\in\mathcal{M}} g_i(x_i^*) + {\bm{\lambda}^*}^T (A\hat{z}^{K+1}-\mathbf{\hat{x}}^{K+1}) \leq \frac{C}{2(K+1)},
\end{align*}
so that 
\begin{align*}
\left\lvert  \sum_{i\in\mathcal{M}} g_i(\hat{x}_i^{K+1}) - \sum_{i\in\mathcal{M}} g_i(x_i^*)\right\rvert \leq \frac{C}{2(K+1)} + \frac{\lVert {\bm{\lambda}^*}\rVert}{(K+1)}  \left\lVert \sum_{k=0}^K (Az^{k+1}-\mathbf{x}^{k+1}) \right\rVert.
\end{align*}
Next,  letting $P$ be a diagonal matrix of appropriate size with elements equal to the $\rho_i$:
\begin{align*}
\left\lVert \sum_{k=0}^K (Az^{k+1}-\mathbf{x}^{k+1}) \right\rVert &= \left\lVert \sum_{k=0}^K P^{-1} \left(\bm{\lambda}^{k+1} - \bm{\lambda}^{k}\right) \right\rVert,\\
&=\left\lVert P^{-1} \left(\bm{\lambda}^{K+1} - \bm{\lambda}^{0}\right) \right\rVert,\\
&\leq \lVert P^{-1}\rVert \left( \left\lVert \bm{\lambda}^{K+1} - \bm{\lambda}^{*} \right\rVert + \left\lVert \bm{\lambda}^{0} - \bm{\lambda}^{*} \right\rVert \right).
\end{align*}
Additionally,  as noted in Remark~\ref{rem:lyapunov},  Proposition~\ref{prop:conv_conj} also implies that $V^k\leq V^0$ for any $k$, and in particular that:
\begin{align*}
\frac{1}{2} \left\lVert \sqrt{P}^{-1} (\bm{\lambda}^{k+1} - \bm{\lambda}^*)\right\rVert^2 \leq \frac{C}{2},
\end{align*}
whence it follows that:
\begin{align*}
 \left\lVert \bm{\lambda}^{k+1} - \bm{\lambda}^*\right\rVert \leq \sqrt{\bar{\rho}C},
\end{align*}
and thus that:
\begin{align*}
\left\lVert A\hat{z}^{k+1}-\mathbf{\hat{x}}^{k+1} \right\rVert &\leq   \frac{2 \sqrt{\bar{\rho} C}}{\underline{\rho}(K+1)},
\end{align*}
which concludes the proof.

\subsection{Proof of Theorem~\ref{thm:linear_conv}}\label{app:proof_linear_conv}

To prove the result, we will  bound each term in $V^k$ by a term involving $V^k-V^{k+1}$.  In order to achieve that, we first derive the inequalities resulting from making use of the strong-convexity and Lipschitz-continuity of the gradients of all the functions.
\begin{description}
\item[Strong convexity bound:]
The bound resulting from the application of the functions' $\mu_i$-strong convexity was already established in Equation~\eqref{eq:main_ineq_strong_conv} of Proposition~\ref{prop:main_ineq} as:
\begin{align*}
\sum_{i\in\mathcal{M}} \frac{\mu_i}{2} \lVert x_i^{k+1}-x_i^*\rVert^2  \leq V^k - V^{k+1} - r^{k+1},
\end{align*}
and in particular:
\begin{align}
\sum_{i\in\mathcal{M}} \frac{\mu_i}{2} \lVert x_i^{k+1}-x_i^*\rVert^2  \leq V^k - V^{k+1},\label{eq:grad_lipsch_bound_0}
\end{align}
\item[Gradient Lipschitz-continuity bound:]
The application of gradient Lipschitz continuity to Equation~\eqref{eq:main_ineq} would have added a term $-\frac{1}{2 L_i}\lVert \nabla g_i(x_i^{k+1})-\nabla g_i(x_i^{*})\rVert^2$ to the right of  Equation~\eqref{eq:sublinear_proof_dummy1} in the proof of Proposition~\ref{prop:main_ineq}, and yielded:
\begin{align}
\sum_{i\in\mathcal{M}}\frac{1}{2L_i}\lVert \nabla g_i(x_i^{k+1})-\nabla g_i(x_i^{*})\rVert^2 \leq V^k - V^{k+1}- r^{k+1}.\label{eq:grad_lipsch_bound_1}
\end{align}
We can carry out a unified treatment of the terms $\frac{1}{2L_i} \lVert \nabla g_i(x_i^{k+1})-\nabla g_i(x_i^{*})\rVert^2$ by recalling that the optimality conditions \eqref{eq:primal_opt_cond}, \eqref{eq:dual_opt_cond}, and \eqref{eq:prox_opt_cond} can all be written as:
\begin{align*}
\nabla g_i(x_i^{k+1})=\lambda_i^k + \tilde{\rho}_i (z^k-x_i^{k+1}) - \left(\nabla \phi_i(x_i^{k+1}) - \nabla \phi_i(x_i^{k}) \right),
\end{align*}
with:
\begin{align*}
\tilde{\rho}_i &= \begin{cases}
\rho_i & i\in\mathcal{P}\cup\mathcal{X},\\
0& i \in \mathcal{D}
\end{cases},&
\phi_i(x)&= 
0,~i\in\mathcal{D}\cup\mathcal{X}.
\end{align*}
We then have:
\begin{align*}
&\frac{1}{2L_i} \lVert \nabla g_i(x_i^{k+1})-\nabla g_i(x_i^{*})\rVert^2\\
&=\frac{1}{2L_i} \left\lVert \lambda_i^k - \lambda_i^* + \tilde{\rho}_i (z^k-x_i^{k+1}) - \left(\nabla \phi_i(x_i^{k+1}) - \nabla \phi_i(x_i^{k}) \right)\right\rVert^2\\
&\geq \frac{(1-\gamma_i)}{2L_i} \left\lVert \lambda_i^k - \lambda_i^* + \tilde{\rho}_i (z^k-x_i^{k+1})\right\rVert^2 -  \frac{(\frac{1}{\gamma_i}-1)}{2L_i}\left\lVert\nabla \phi_i(x_i^{k+1}) - \nabla \phi_i(x_i^{k}) \right\rVert^2\\
&= \frac{1}{2(L_i+L_{\phi_i})} \left\lVert \lambda_i^k - \lambda_i^* + \tilde{\rho}_i (z^k-x_i^{k+1})\right\rVert^2 -  \frac{1}{2L_{\phi}}\left\lVert\nabla \phi_i(x_i^{k+1}) - \nabla \phi_i(x_i^{k}) \right\rVert^2\\
&\geq \frac{1}{2(L_i+L_{\phi_i})} \left\lVert \lambda_i^k - \lambda_i^* + \tilde{\rho}_i (z^k-x_i^{k+1})\right\rVert^2 -  D_{\phi_i}(x_i^{k+1},x_i^k).
\end{align*}
We here used the fact that $\lVert u + v \rVert^2 \geq (1-\gamma) \lVert u \rVert^2 - \left(\frac{1}{\gamma}-1\right) \lVert v\rVert^2$ (see Lemma~\ref{lem:sum_square}) and set  $\gamma_i =\frac{L_{\phi_i}}{L_i+L_{\phi_i}}$.  This value was set so as to cancel the Bregman divergence term between $x_i^{k+1}$ and $x_i^k$ with the one present in $r^{k+1}$.
We then proceed similarly to isolate the term involving $\lVert \lambda_i^k -\lambda_i^*\rVert^2$:
\begin{align*}
&\frac{1}{2L_i} \lVert \nabla g_i(x_i^{k+1})-\nabla g_i(x_i^{*})\rVert^2\\
&\geq \frac{1}{2(L_i+L_{\phi_i})} \left\lVert \lambda_i^k - \lambda_i^* + \tilde{\rho}_i (z^k-x_i^{k+1})\right\rVert^2 -  D_{\phi_i}(x_i^{k+1},x_i^k)\\
&\geq \frac{(1-\nu_i)}{2(L_i+L_{\phi_i})} \left\lVert \lambda_i^k - \lambda_i^*\right\rVert^2 - \frac{(\frac{1}{\nu_i}-1)}{2(L_i+L_{\phi_i})}  \tilde{\rho}_i^2 \left\lVert(z^k-x_i^{k+1})\right\rVert^2 -  D_{\phi_i}(x_i^{k+1},x_i^k)\\
&= \frac{1}{2(L_i+L_{\phi_i}+\rho_i)} \left\lVert \lambda_i^k - \lambda_i^*\right\rVert^2 - \frac{ \tilde{\rho}_i}{2}  \left\lVert(z^k-x_i^{k+1})\right\rVert^2 -  D_{\phi_i}(x_i^{k+1},x_i^k),
\end{align*}
where we again used Lemma~\ref{lem:sum_square} and set $\nu_i=\frac{\tilde{\rho}_i}{L_i+L_{\phi_i}+\rho_i}$ in order to cancel the $\lVert z^k - x_i^{k+1}\rVert^2$ with the one in $r^{k+1}$.  Plugging this back into Equation~\eqref{eq:grad_lipsch_bound_1} yields:
\begin{align*}
&\sum_{i\in\mathcal{M}} \frac{1}{2(L_i+L_{\phi_i}+\tilde{\rho}_i)} \left\lVert \lambda_i^k - \lambda_i^*\right\rVert^2 - \frac{ \tilde{\rho}_i}{2}  \left\lVert(z^k-x_i^{k+1})\right\rVert^2 -  D_{\phi_i}(x_i^{k+1},x_i^k)\\
&\leq \sum_{i\in\mathcal{M}}\frac{1}{2L_i}\lVert \nabla g_i(x_i^{k+1})-\nabla g_i(x_i^{*})\rVert^2 \\
&\leq V^k - V^{k+1}- r^{k+1}
\end{align*}
which implies
\begin{align}
\sum_{i\in\mathcal{M}} \frac{1}{2(L_i+L_{\phi_i}+\tilde{\rho}_i)} \left\lVert \lambda_i^k - \lambda_i^*\right\rVert^2 \leq V^k-V^{k+1}.\label{eq:grad_lipsch_bound_2}
\end{align}
\end{description}

We now bound each term in $V^{k+1}$. 
\begin{itemize}
\item Using \eqref{eq:grad_lipsch_bound_2},  we have:
\begin{align*}
\frac{\underline{\rho}_\mathcal{M}}{\alpha_\mathcal{M}} \sum_{i\in\mathcal{M}} \frac{1}{2\rho_i}\lVert \lambda_i^{k+1} - \lambda_i^*\rVert^2   \leq \sum_{i\in\mathcal{M}} \frac{1}{2(L_i+L_{\phi_i}+\tilde{\rho}_i)} \left\lVert \lambda_i^{k+1} - \lambda_i^*\right\rVert^2 \leq V^{k+1}-V^{k+2}.
\end{align*}
\item Recall that since $-g_i^*$ has $\frac{1}{\mu_i}$-Lipschitz continuous gradients, we have according to Lemma~\ref{lem:bregman_lipschitz} that $D_{-g_i^*}(\lambda_i^{k+1},\lambda_i^*)\leq \frac{1}{2\mu_i}\lVert \lambda_i^{k+1} - \lambda_i^*\rVert^2$, and so, using \eqref{eq:grad_lipsch_bound_2} once again:
\begin{align*}
\frac{\underline{\mu}_\mathcal{D}}{\alpha_\mathcal{D}} \sum_{i\in\mathcal{D}} D_{-g_i^*}(\lambda_i^{k+1},\lambda_i^*) \leq V^{k+1}-V^{k+2}.
\end{align*}
\item Using Jensen's inequality and \eqref{eq:grad_lipsch_bound_0}, we have:
\begin{align*}
\frac{\underline{\mu}_{\mathcal{P}\cup\mathcal{X}}}{\overline{\rho}_{\mathcal{P}\cup\mathcal{X}}}\sum_{i\in\mathcal{P}\cup\mathcal{X}} \frac{\rho_i}{2} \lVert z^{k+1}-z^*\rVert^2 &\leq \frac{\underline{\mu}_{\mathcal{P}\cup\mathcal{X}}}{\overline{\rho}_{\mathcal{P}\cup\mathcal{X}}}\sum_{i\in\mathcal{M}} \frac{\rho_i}{2} \lVert z^{k+1}-z^*\rVert^2\\
&\leq \frac{\underline{\mu}_{\mathcal{P}\cup\mathcal{X}}}{\overline{\rho}_{\mathcal{P}\cup\mathcal{X}}}\sum_{i\in\mathcal{M}} \frac{\rho_i}{2} \lVert x_i^{k+1}-x_i^*\rVert^2\\
&\leq \sum_{i\in\mathcal{M}} \frac{\mu_i}{2} \lVert x_i^{k+1}-x_i^*\rVert^2\\
&\leq  V^k - V^{k+1}.
 \end{align*}
 \item Finally, using Lemma~\ref{lem:bregman_lipschitz}, we have $D_{\phi_i}(x_i^*,x_i^{k+1})\leq \frac{L_{\phi_i}}{2}\lVert x_i^{k+1}-x_i^*\rVert^2$, which combined with \eqref{eq:grad_lipsch_bound_0} gives:
 \begin{align*}
 \frac{\underline{\mu}_\mathcal{P}}{\overline{L_\phi}}\sum_{i\in\mathcal{P}} D_{\phi_i}(x_i^*,x_i^{k+1}) 
 &\leq  \frac{\underline{\mu}_\mathcal{P}}{\overline{L_\phi}}\sum_{i\in\mathcal{P}} \frac{L_{\phi_i}}{2}\lVert x_i^* -x_i^{k+1}\rVert^2\\
 &\leq  \sum_{i\in\mathcal{P}} \frac{\mu_i}{2}\lVert x_i^*-x_i^{k+1}\rVert^2\\
 &\leq \sum_{i\in\mathcal{M}} \frac{\mu_i}{2}\lVert x_i^*-x_i^{k+1}\rVert^2\\
 &\leq V^k - V^{k+1}.
 \end{align*}
\end{itemize}
Grouping the four derived inequalities, we have:
\begin{align*}
\frac{\underline{\rho}_{\mathcal{M}}}{\alpha_{\mathcal{M}}} \sum_{i\in\mathcal{M}} \frac{1}{2\rho_i}\lVert \lambda_i^{k+1} - \lambda_i^*\rVert^2  &\leq V^{k+1}-V^{k+2},\\
\frac{\underline{\mu}_{\mathcal{D}}}{\alpha_{\mathcal{D}}} \sum_{i\in\mathcal{D}} D_{-g_i^*}(\lambda_i^{k+1},\lambda_i^*) &\leq V^{k+1}-V^{k+2}\\
\frac{\underline{\mu}_{\mathcal{P}\cup\mathcal{X}}}{\overline{\rho}_{\mathcal{P}\cup\mathcal{X}}}\sum_{i\in\mathcal{P}\cup\mathcal{X}} \frac{\rho_i}{2} \lVert z^{k+1}-z^*\rVert^2 &\leq  V^k - V^{k+1}\\
 \frac{\underline{\mu}_{\mathcal{P}}}{\overline{L_\phi}}\sum_{i\in\mathcal{P}} D_{\phi_i}(x_i^*,x_i^{k+1})  &\leq V^k - V^{k+1}.
\end{align*}
Summing these four inequalities yields:
\begin{align*}
\frac{1}{2}\min\left\{\frac{\underline{\rho}_\mathcal{M}}{\alpha_\mathcal{M}}, \frac{\underline{\mu}_\mathcal{D}}{\alpha_\mathcal{D}} , \frac{\underline{\mu}_{\mathcal{P}\cup\mathcal{X}}}{\overline{\rho}_{\mathcal{P}\cup\mathcal{X}}},  \frac{\underline{\mu}_\mathcal{P}}{\overline{L_\phi}}\right\} V^{k+1}\leq V^k - V^{k+2}.
\end{align*}
The monotonicity of the sequence $\{V^k\}$ (see Remark~\ref{rem:lyapunov}) further implies that $V^{k+2}\leq V^{k+1}$, and thus~that:
\begin{align*}
\frac{1}{2}\min\left\{\frac{\underline{\rho}_\mathcal{M}}{\alpha_\mathcal{M}}, \frac{\underline{\mu}_\mathcal{D}}{\alpha_\mathcal{D}} , \frac{\underline{\mu}_{\mathcal{P}\cup\mathcal{X}}}{\overline{\rho}_{\mathcal{P}\cup\mathcal{X}}},  \frac{\underline{\mu}_\mathcal{P}}{\overline{L_\phi}}\right\} V^{k+2}\leq V^k - V^{k+2},
\end{align*}
whence:
\begin{align*}
 V^{k+2} \leq \left(1 + \frac{1}{2}\min\left\{\frac{\underline{\rho}_\mathcal{M}}{\alpha_\mathcal{M}}, \frac{\underline{\mu}_\mathcal{D}}{\alpha_\mathcal{D}} , \frac{\underline{\mu}_{\mathcal{P}\cup\mathcal{X}}}{\overline{\rho}_{\mathcal{P}\cup\mathcal{X}}},  \frac{\underline{\mu}_\mathcal{P}}{\overline{L_\phi}}\right\}\right)^{-1} V^{k}.
\end{align*}

\subsection{Proof of Corollary~\ref{cor:linear_conv}}\label{app:proof_linear_conv_corollary}

\begin{description}
\item[All Dual Agents:]
If all the agents are dual, then $V^k = \sum_{i\in\mathcal{D}} \frac{1}{2\rho_i}\|\lambda_i^{k} - \lambda_i^*\|^2 + D_{-g_i^*}(\lambda_i^{k},\lambda_i^*)$. Summing up only the first two inequalities yields:
\begin{align*}
\frac{1}{2}\min\left\{\frac{\underline{\rho}_\mathcal{D}}{\alpha_\mathcal{D}}, \frac{\underline{\mu}_\mathcal{D}}{\alpha_\mathcal{D}} \right\} V^{k}\leq V^k - V^{k+1}.
\end{align*}
The monotonicity of the sequence $\{V^k\}$ (see Remark~\ref{rem:lyapunov}) further implies that $V^{k}\leq V^{k+1}$, and thus~that:
\begin{align*}
\frac{1}{2}\min\left\{\frac{\underline{\rho}_\mathcal{M}}{\alpha_\mathcal{M}}, \frac{\underline{\mu}_\mathcal{D}}{\alpha_\mathcal{D}} \right\} V^{k+1}\leq V^k - V^{k+1},
\end{align*}
whence:
\begin{align*}
 V^{k+1} \leq \left(1 + \frac{1}{2}\min\left\{\frac{\underline{\rho}_\mathcal{M}}{\alpha_\mathcal{M}}, \frac{\underline{\mu}_\mathcal{D}}{\alpha_\mathcal{D}} \right\}\right)^{-1} V^{k}.
\end{align*}
\item[All Primal Agents:]
When all the agents are primal (or proximal), we recall from Remark~\ref{rem:rk} that $r^{k}$ reads:
\begin{align*}
r^k&=  \sum_{i\in\mathcal{M}} \frac{1}{2\rho_i} \lVert \lambda_i^k-\lambda_i^{k-1}\rVert^2 + \sum_{i\in\mathcal{M}} \frac{\rho_i}{2} \lVert z^{k} - z^{k-1}\rVert^2 + \sum_{i\in\mathcal{P}} D_{\phi_i}(x_i^{k}, x_i^{k-1}).
\end{align*}
As a result,  we may change the bounds used in the Gradient Lipshitz-continuity bound in proof \ref{app:proof_linear_conv} as follows:
\begin{align*}
&\frac{1}{2L_i} \lVert \nabla g_i(x_i^{k+1})-\nabla g_i(x_i^{*})\rVert^2\\
&\geq \frac{1}{2(L_i+L_{\phi_i})} \left\lVert \lambda_i^{k+1} - \lambda_i^* + \tilde{\rho}_i (z^k-z^{k+1})\right\rVert^2 -  D_{\phi_i}(x_i^{k+1},x_i^k)\\
&\geq \frac{(1-\nu_i)}{2(L_i+L_{\phi_i})} \left\lVert \lambda_i^{k+1} - \lambda_i^*\right\rVert^2 - \frac{(\frac{1}{\nu_i}-1)}{2(L_i+L_{\phi_i})}  \tilde{\rho}_i^2 \left\lVert(z^k-z^{k+1})\right\rVert^2 -  D_{\phi_i}(x_i^{k+1},x_i^k)\\
&= \frac{1}{2(L_i+L_{\phi_i}+\rho_i)} \left\lVert \lambda_i^{k+1} - \lambda_i^*\right\rVert^2 - \frac{ \tilde{\rho}_i}{2}  \left\lVert(z^k-z^{k+1})\right\rVert^2 -  D_{\phi_i}(x_i^{k+1},x_i^k),
\end{align*}
where we used Lemma~\ref{lem:sum_square} and set $\nu_i=\frac{\tilde{\rho}_i}{L_i+L_{\phi_i}+\rho_i}$ in order to cancel the $\lVert z^k - z^{k+1}\rVert^2$ with the one in $r^{k+1}$. 
As a result, we may replace \eqref{eq:grad_lipsch_bound_2} with:
\begin{align*}
\sum_{i\in\mathcal{P}} \frac{1}{2(L_i+L_{\phi_i}+\tilde{\rho}_i)} \left\lVert \lambda_i^{k+1} - \lambda_i^*\right\rVert^2 \leq V^k-V^{k+1}.
\end{align*}
Proceeding exactly as in  \ref{app:proof_linear_conv}, we obtain the three following inequalities:
\begin{align*}
\frac{\underline{\rho}_{\mathcal{P}}}{\alpha_{\mathcal{P}}} \sum_{i\in\mathcal{P}} \frac{1}{2\rho_i}\lVert \lambda_i^{k+1} - \lambda_i^*\rVert^2  &\leq V^{k}-V^{k+1},\\
\frac{\underline{\mu}_{\mathcal{P}}}{\overline{\rho}_{\mathcal{P}}}\sum_{i\in\mathcal{P}} \frac{\rho_i}{2} \lVert z^{k+1}-z^*\rVert^2 &\leq  V^k - V^{k+1}\\
 \frac{\underline{\mu}_{\mathcal{P}}}{\overline{L_\phi}}\sum_{i\in\mathcal{P}} D_{\phi_i}(x_i^*,x_i^{k+1})  &\leq V^k - V^{k+1}.
\end{align*}
Summing up these three inequalities yields the desired result.
\item[All Proximal Agents:]
The proof is identical to the one for all proximal agents, except that there is one fewer inequality due to the absence of the Bregman terms.

\end{description}

\section{Bregman Divergence}

\begin{definition}[Bregman Divergence]\label{def:bregman}
Given a differentiable convex function $\phi$, the associated Bregman divergence is defined as:
\begin{align*}
D_\phi(y,x) &= \phi(y) - \phi(x) - \nabla \phi(x)^T (y-x).
\end{align*}
\end{definition}

\begin{lemma}\label{lem:bregman}
Let $\phi$ be a convex function and $D_\phi$ be the associated Bregman divergence. Then:
\begin{align*}
\left(\nabla \phi (u) - \nabla \phi(v)\right)^T (w-u) = D_\phi(w,v) - D_\phi(w,u) - D_\phi(u,v),\quad \forall u,v,w.
\end{align*}
\end{lemma}

\begin{lemma}\label{lem:bregman_lipschitz}
Let $\phi$ be a convex function with $L$-Lipschitz continuous gradients, and $D_\phi$ be the associated Bregman divergence. Then:
\begin{align*}
D_{\phi}(x,y)\leq \frac{L}{2}\lVert x - y\rVert^2.
\end{align*}
\end{lemma}

\section{Convex Conjugate}\label{app:conv_conj}

\begin{definition}[Convex Conjugate]\label{def:conv_conj}
Let $f:\mathbb{R}^n\to \mathbb{R}$ be a convex function. Its convex conjugate $f^*:\mathbb{R}^n\to \mathbb{R}$ is defined as:
\begin{align*}
f^*(y)&=\sup_{x\in\mathbb{R}^n} \left\{x^Ty-f(x)\right\}.
\end{align*}
\end{definition}

\begin{proposition}\label{prop:conv_conj}
Let $f$ be a convex and differentiable function, and $f^*$ be its convex conjugate. Then the following statements are equivalent:
\begin{align*}
y=\nabla f(x) \Leftrightarrow x=\nabla f^*(y)\Leftrightarrow x^T y = f(x) + f^*(y).
\end{align*}
\end{proposition}

\section{Inequalities}

\begin{lemma}\label{lem:sum_square}
For any two vectors $x,y\in\mathbb{R}^n$,  and $\nu>0$, we have:
\begin{align*}
\lVert x - y \rVert^2 \geq (1-\nu) \lVert x\rVert^2 - \left(\frac{1}{\nu}-1\right) \lVert y\rVert^2.
\end{align*}
\end{lemma}

\begin{proof}
We have:
\begin{align*}
\lVert x+ y \rVert^2 &= \lVert x\rVert^2 + \lVert y\rVert^2 + 2 x^Ty\\
&= \lVert x\rVert^2 + \lVert y\rVert^2 + 2 (\sqrt{\nu} x)^T\left(\frac{1}{\sqrt{\nu}}\right) y\\
&= \lVert x\rVert^2 + \lVert y\rVert^2 + \left\lVert \sqrt{\nu} x + \frac{1}{\sqrt{\nu}} y\right\rVert^2 - \nu \lVert x\rVert^2 - \frac{1}{\nu}\lVert y \rVert^2\\
 &\geq (1-\nu) \lVert x\rVert^2 - \left(\frac{1}{\nu}-1\right) \lVert y\rVert^2.
\end{align*}
\end{proof}

\end{document}